\documentclass[a4paper]{article} 
\usepackage{lineno}
\usepackage{todonotes}
\usepackage{amsfonts}
\usepackage{amsmath}
\usepackage{amsthm}
\usepackage{graphicx}
\usepackage{subcaption}
\usepackage{etoolbox}
\usepackage{xargs}
\usepackage[pdftex,dvipsnames]{xcolor}  

\makeatletter

\newtheorem{theorem}{Theorem}

\newtheorem{lemma}[theorem]{Lemma}

\newcommand{\changeoperator}[1]{%
  \csletcs{#1@saved}{#1@}%
  \csdef{#1@}{\changed@operator{#1}}%
}
\newcommand{\changed@operator}[1]{%
  \mathop{%
    \mathchoice{\textstyle\csuse{#1@saved}}
               {\csuse{#1@saved}}
               {\csuse{#1@saved}}
               {\csuse{#1@saved}}%
  }%
}

\makeatother
\changeoperator{prod}

\newcommand{\R}{\mathbb{R}}
\newcommand{\C}{\mathbb{C}}
\newcommand{\Z}{\mathbb{Z}}
\newcommand{\N}{\mathbb{N}}
\newcommand{\bv}{\bar{v}}
\newcommand{\bydef}{\,\stackrel{\mbox{\tiny\textnormal{\raisebox{0ex}[0ex][0ex]{def}}}}{=}\,} 
    
\newcommand{\proj}[2]{ \prod \limits_{\scriptscriptstyle{#1}}^{\mathsf{\scriptscriptstyle{#2}}}}
\newcommand{\sF}{\scriptscriptstyle{\mathsf{F}}}
\newcommand{\sTF}{\scriptscriptstyle{\mathsf{TF}}}
\newcommand{\sC}{\scriptscriptstyle{\mathsf{C}}}
\newcommand{\sT}{\scriptscriptstyle{\mathsf{T}}}

\newcommand{\Lr}{\mathsf{L}}

\DeclareMathOperator{\fop}{\mathit{f}}

\newif\ifchanges

\newcommand{\add}[1]{\ifchanges\textcolor{blue}{#1}\else#1\fi}

\definecolor{tealink}{RGB}{21, 154, 137}
\definecolor{BrickRed}{rgb}{0.6, 0.2, 0.2}

\usepackage{hyperref}
\hypersetup{
   pdftitle={Rigorous validation of gap solitons in Bose-Einstein condensates},
   pdfauthor={},
   pdfsubject={},
   pdfkeywords={},
   bookmarksnumbered=true,     
   bookmarksopen=true,         
   bookmarksopenlevel=1,       
   colorlinks=true,         
   allcolors=tealink,
   citecolor=BrickRed,
   pdfstartview=Fit,           
   pdfpagemode=UseOutlines,    
   pdfpagelayout=TwoPageRight
}

\begin{document}

\title{Computer-Assisted Proofs of Gap Solitons in Bose-Einstein Condensates }
\author{
Miguel 
Ayala
\thanks
{McGill University, Department of Mathematics and Statistics, 805 Sherbrooke Street West, Montreal, QC, H3A 0B9, Canada. {\tt miguel.ayala@mail.mcgill.ca}.}
\and
Carlos Garc\'{i}a-Azpeitia
\thanks
{Departamento de Matem\'aticas y Mec\'anica,
    IIMAS-UNAM.
    Apdo. Postal 20-126, Col. San \'Angel,
    Mexico City, 01000,  Mexico. {\tt cgazpe@aries.iimas.unam.mx}.}
\and
Jean-Philippe Lessard
\thanks
{McGill University, Department of Mathematics and Statistics, 805 Sherbrooke Street West, Montreal, QC, H3A 0B9, Canada. {\tt jp.lessard@mcgill.ca}.}
}

\maketitle

\begin{abstract}
We provide a framework for turning a numerical simulation of a gap soliton in the one-dimensional Gross-Pitaevskii equation into a rigorous mathematical proof of its existence. These nonlinear localized solutions play a central role in the study of Bose-Einstein condensates (BECs). We reformulate the problem of proving their existence as the search for homoclinic orbits in a dynamical system. We then apply computer-assisted proof techniques to obtain verifiable conditions under which a numerically approximated trajectory corresponds to a true homoclinic orbit. This work also presents the first examples of computer-assisted proofs of gap solitons in the Gross-Pitaevskii equation on non-perturbative parameter regimes. 
\end{abstract}   

\begin{center}
\end{center}

\section{Introduction} \label{sec:intro}
A Bose--Einstein condensate (BEC) is a state of matter that occurs when a
collection of particles cools down to temperatures near absolute zero, causing
them to lose their individual identities and behave as a single 
wave. The first experimental realization of BECs at ultra-cold temperatures earned the 2001 Nobel Prize in Physics, and since then, Bose-Einstein condensates have provided a platform for exploring quantum mechanics on large scales, with applications in precision measurements, quantum computing, and the modeling of complex systems such as superfluidity and optical lattices.
For a comprehensive review of both experimental and theoretical developments on Bose--Einstein condensates, see \cite{Lous}.

Beyond their physical significance, the study of
Bose--Einstein condensates offers fertile ground for advancing theoretical methods in nonlinear dynamics and partial
differential equations. In this paper, we study the dynamics of a BEC 
using the time-dependent Gross--Pitaevskii equation, which
models the BEC's evolution in one spatial dimension:
\begin{equation}  \label{eq:time-dependent-GP}
i \partial_{t} \psi = -\partial_{x}^{2}\psi+V(x)\psi + c |\psi|^{2}\psi.
\end{equation}
Here, $\psi(t,x) \in \mathbb{C}$ denotes the dimensionless wave function, $%
|\psi|^{2}$ represents the BEC density, and $V(x)$ is the external potential
created by the optical lattice along the spatial domain $x \in \mathbb{R}$.
The book \cite{Kevrekidis2008EmergentNP} provides a detailed treatment of the Gross--Pitaevskii equation and nonlinear wave dynamics, combining experimental perspectives with numerical studies of Bose--Einstein condensates. It also includes an extensive bibliography covering many areas of the field.

A fundamental aspect of Bose-Einstein condensate analysis is the study of standing wave solutions to the Gross--Pitaevskii  equation. These solutions take the form \( \psi(t,x) = e^{-i a t} u(x) \), where \( u(x) \) is a real-valued function that satisfies the time-independent (GP) equation:
\begin{equation}  \label{eq:GP_ode}
(\partial_x^2 + a - V(x)) u - c u^3 = 0.
\end{equation}
Among the nonlinear structures admitted by the GP equation are solitons, nonlinear Bloch waves, and domain walls, each of which has been extensively studied. Our interest lies in \emph{gap solitons}, a class of localized solutions. More specifically, a soliton is a real-valued function \( u : \mathbb{R} \to \mathbb{R} \) that decays to zero at infinity:
\begin{equation}  \label{eq:homoclinic_conditions}
\lim_{x \to \pm \infty} \left( u(x), u'(x) \right) = (0, 0).
\end{equation}
Chapter 19 of \cite{Kevrekidis2008EmergentNP} surveys the existence of solitons under various potential types, with particular emphasis on periodic potentials. In the present work, we study the specific case \( V(x) = b \cos(2x) \), which models the dynamics of a Bose-Einstein condensate in an optical lattice. To understand the setting in which gap solitons arise, we consider the linearization of the Gross--Pitaevskii equation around the trivial solution, known as Mathieu's equation:
\begin{equation}  \label{eq:Mathieu_ODE}
L u \,\overset{\mbox{\tiny\textnormal{\raisebox{0ex}[0ex][0ex]{def}}}}{=}\,
\left( \partial_x^2 + a - V(x) \right) u = \left( \partial_x^2 + a - b
\cos(2x) \right) u = 0.
\end{equation}
Bloch theory predicts that the spectrum of this linear operator consists of bands separated by spectral gaps. The edges of these gaps are determined by solutions to \eqref{eq:Mathieu_ODE}, known as \emph{Mathieu functions}. In the purely linear setting, soliton-type solutions cannot exist. When nonlinearity is introduced, however, localized modes—\emph{gap solitons}—can form within the spectral gaps. These structures are characteristic of nonlinear wave systems and closely resemble gap solitons observed in nonlinear optics (e.g., see \cite{Pelinosvsky}).

Researchers have frequently used perturbative asymptotics and numerical simulations to investigate soliton solutions. For example, \cite{Pelinosvsky} applies asymptotic analytical methods to compute gap solitons in all spectral gaps of a periodic potential, showing that these solitons bifurcate from distinct band edges depending on the sign of \( c \). Moreover, \cite{Alfimov_2002} introduces a numerical approach for calculating gap solitons in the repulsive case (\( c = 1 \)). The study shows that the subset of non-blow-up solutions, which includes all gap solitons, has a fractal structure in the space of initial conditions. This fractal structure makes it possible to identify gap solitons over large regions in the parameter space \( (a, b) \in \mathbb{R}^2 \). However, \emph{while simulations and asymptotic techniques provide valuable insight, they do not provide a rigorous proof of existence. Abstract results, on the other hand, often lack any explicit description of the solution profiles.}


The structure of the problem, however, makes it well-suited for a computer-assisted proof (CAP) approach. In this work, we derive verifiable conditions under which a soliton exists near a given numerical approximation. When these conditions hold, our approach guarantees a rigorous proof of existence and supplies tight, explicit \( C^0 \) error bounds for the discrepancy between the exact solution and its approximation. It is important to emphasize that our method can be applied
to any set of parameters $a$, $b$ and $c$ for which an accurate numerical soliton solution is
available. For example, Figure~\ref{fig:Parts_of_proof} shows a nontrivial numerical approximation of an even soliton solution to equation \eqref{eq:GP_ode}, as appearing in \cite{Alfimov2000}. Our method provides a way to rigorously validate such approximations:
\begin{theorem}
\label{theorem:main}
   The Gross-Pitaevskii equation \eqref{eq:GP_ode} with parameters \( a = 1.1025 \), \( b = 0.55125 \), and \( c = -0.826875 \) has a soliton solution $u:\R \to \R$,
    satisfying
    \[
    \|u - \bar{u}\|_{\infty} \leq 8.617584260554394 \cdot 10^{-6},
    \]
    where \( \bar{u} \) is a numerical approximation of the solution illustrated in Figure~\ref{fig:Parts_of_proof}.
\end{theorem}

\begin{figure}[!ht]
\centering
\includegraphics[width=12cm, trim=0 0 0 1.5cm, clip]{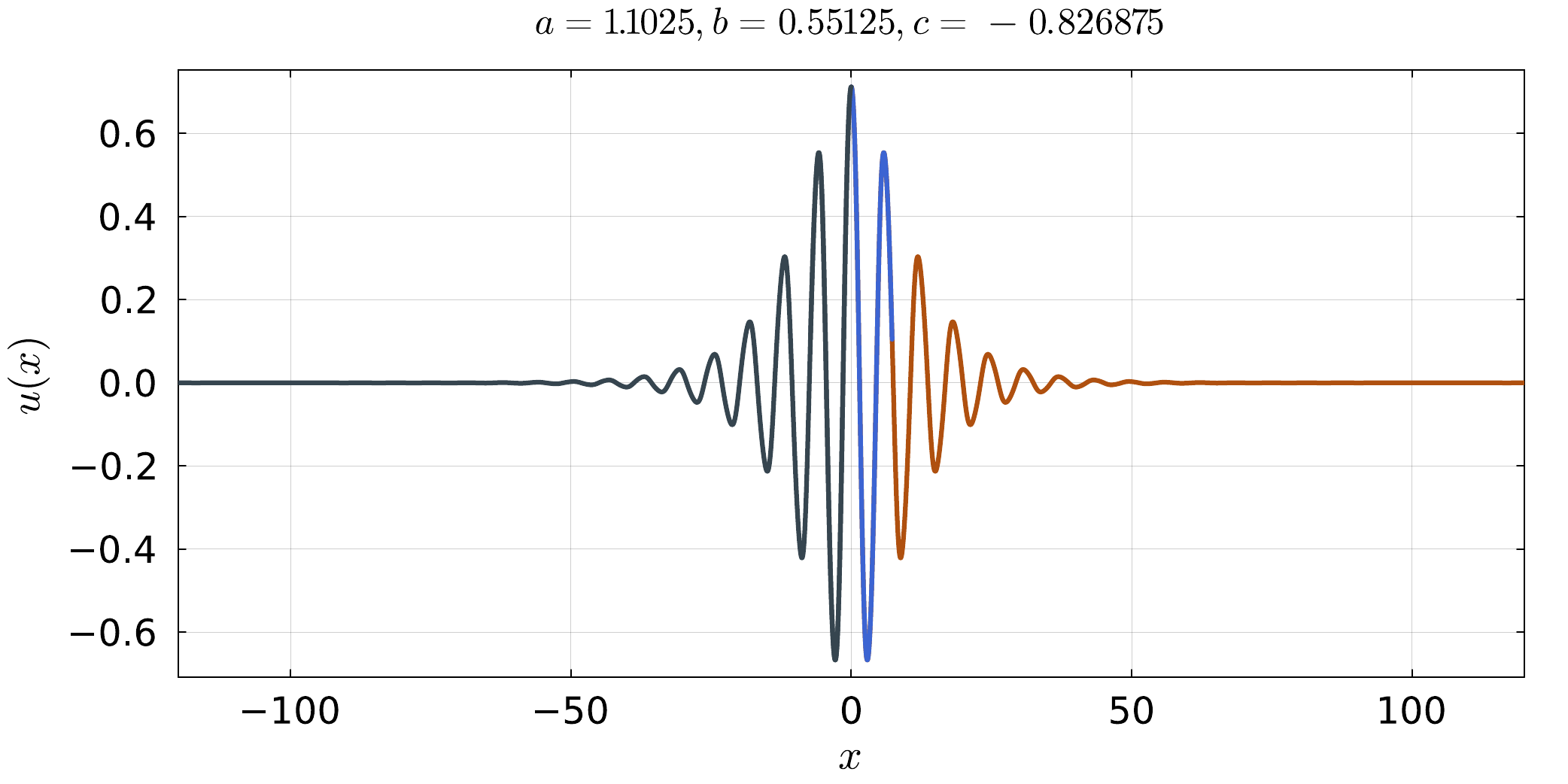}
\caption{
The figure shows a validated soliton solution of the Gross-Pitaevskii (GP) equation with parameters 
\( a = 1.1025 \), \( b = 0.55125 \), and \( c = -0.826875 \). 
It also depicts the main elements of our approach: the solution to the boundary-value problem (blue), 
the stable manifold (orange), and the even extension of the soliton (black).
}
\label{fig:Parts_of_proof}
\end{figure}


To describe our approach, we begin by reformulating the problem using a standard dynamical systems framework. In this setting, finding a gap soliton becomes the search for a connecting orbit between invariant sets.
More precisely, denote $u_{1}\,\overset{\mbox{\tiny\textnormal{%
\raisebox{0ex}[0ex][0ex]{def}}}}{=}\,u$, $u_{2}\,\overset{%
\mbox{\tiny\textnormal{\raisebox{0ex}[0ex][0ex]{def}}}}{=}\,u^{\prime }$, $%
u_{3}\,\overset{\mbox{\tiny\textnormal{\raisebox{0ex}[0ex][0ex]{def}}}}{=}%
\,V(x)=\cos (2x)$ (the periodic potential) and $u_{4}\,\overset{%
\mbox{\tiny\textnormal{\raisebox{0ex}[0ex][0ex]{def}}}}{=}\,u_{3}^{\prime
}=-2\sin (2x)$. Note that $u_{3}$ solves $u_{3}^{\prime \prime }=-4u_{3}$
with initial conditions $(u_{3}(0),u_{3}^{\prime }(0))=(1,0)$. 
This formulation allows us to transform equation $\eqref{eq:GP_ode}$ into a
autonomous polynomial vector field.
Indeed, assume that $%
U\,\overset{\mbox{\tiny\textnormal{\raisebox{0ex}[0ex][0ex]{def}}}}{=}%
\,(u_{1},u_{2},u_{3},u_{4})\in \mathbb{R}^{4}$ is a solution to the
four-dimensional autonomous system 
\begin{equation}
\frac{dU}{dx}=g(U)\,\overset{\mbox{\tiny\textnormal{%
\raisebox{0ex}[0ex][0ex]{def}}}}{=}\,%
\begin{pmatrix}
u_{2} \\ 
-au_{1}+bu_{3}u_{1}+cu_{1}^{3} \\ 
u_{4} \\ 
-4u_{3}%
\end{pmatrix}
\label{eq:4D_system}
\end{equation}%
with initial conditions $u_{3}(0)=1$ and $u_{4}(0)=0$. Then, if $%
u_{1},u_{2}\neq 0$ and the conditions in \eqref{eq:homoclinic_conditions}
are satisfied, the first component $u_{1}$ of $U$ is a gap soliton of the GP
equation \eqref{eq:GP_ode}. 

We emphasize that the system \eqref{eq:4D_system} is not conservative, though it possesses the conserved quantity \( H = \frac{1}{2}(u_4^2 + 4u_3^2) \).
This comes from the fact
that  $(u_{3}(x),u_{4}(x))=(V(x),V^{\prime }(x))$ solves the
Hamiltonian system $u_{3}^{\prime }=\partial _{u_{4}}H(u_{3},u_{4})$, $%
u_{4}^{\prime }=-\partial _{u_{3}}H(u_{3},u_{4})$. \add{Our method can be  extended to periodic solutions $V(x)$, where 
$(u_3,u_4) = (V,V')$ satisfies a general polynomial Hamiltonian system with 
Hamiltonian $H(u_3,u_4)$. For analytic Hamiltonians that are not polynomial, 
recently developed techniques based on the discrete Fourier transform could 
be adapted to our setting to achieve similar results (see, e.g., 
\cite{MR3709329,MR4930548}). Thus, although we have implemented our procedure 
specifically for the prototype case of the cosine potential, it is applicable 
to a broad class of other periodic potentials.}


Now, define the set 
\begin{equation}
\gamma \,\overset{\mbox{\tiny\textnormal{\raisebox{0ex}[0ex][0ex]{def}}}}{=}%
\,\{(0,0,\cos (2x),-2\sin (2x)):x\in \lbrack 0,\pi )\}\subset \mathbb{R}^{4},
\label{eq:periodic_orbit}
\end{equation}%
which represents a periodic orbit of the system \eqref{eq:4D_system}.
Looking for a gap soliton reduces to finding a solution $U:\mathbb{R}%
\rightarrow \mathbb{R}^{4}$ of \eqref{eq:4D_system} defined on all of $%
\mathbb{R}$, such that 
\begin{equation*}
\{U(x):x\in \mathbb{R}\}\subset W^{u}(\gamma )\cap W^{s}(\gamma ),
\end{equation*}%
where $W^{u}(\gamma )$ and $W^{s}(\gamma )$ denote the unstable and stable
manifolds of $\gamma $, respectively. In other words, identifying a gap
soliton amounts to finding a homoclinic orbit associated with the periodic orbit $\gamma $ in the four-dimensional system  \eqref{eq:4D_system} .

The study of connecting orbits in dynamical systems is a vast and active area of research that intersects with diverse mathematical disciplines such as algebraic topology, Morse homology, celestial mechanics, chaos theory, and the calculus of variations. Over the years, a variety of mathematical techniques have been developed to address the theoretical challenges associated with these orbits. For instance, perturbative methods within variational frameworks \cite{Berti, Buffoni, Ambrosetti} and non-perturbative techniques \cite{Tanaka, Hofer, MR1070929} have been employed to prove the existence of homoclinic orbits in conservative and Hamiltonian systems. More recently, the advent of computer-assisted proofs in nonlinear analysis has greatly enriched the study of connecting orbits, employing techniques that combine the strengths of topology, functional analysis, and scientific computing. Prominent methodologies include topological covering relations 
\cite{MR2173545, MR1961956}, the Parameterization Method in a functional-analytic setting \cite{MR2821596, MR3353132, VanDenBerg2018,
MR4658475}, homotopy methods \cite{MR2220064}, and interpolation-based
techniques \cite{MR2408123, MR1661847}. These methods have been instrumental
in proving the existence of connecting orbits in the context of ordinary
differential equations (ODEs) and continue to play an important role in advancing the field. In this paper, we contribute to this line of work.

Let $W_{loc}^{s}(\gamma) $ denote the local stable manifold of $\gamma $. By restricting
to the class of even solitons, we impose that $u_1^{\prime }(0)=0$, and
hence, we only need to solve for $x \ge 0$. The asymptotic condition at $x
\to \infty$ for a soliton solution \eqref{eq:homoclinic_conditions} is
reinterpreted as the condition that $U(\mathsf{L}) \in W_{loc}^{s}(\gamma)$,
for some $\mathsf{L} \in \mathbb{R}$. From this reformulation, we deduce
that the first component $u_{1}:[0,\mathsf{L}] \to \mathbb{R}$ of a solution 
$U:[0,\mathsf{L}] \to \mathbb{R}^{4}$ of the boundary-value problem (BVP) 
\begin{equation}  \label{eq:BVP_soliton}
\dot U(x) = g(U(x)), \quad x \in [0,\mathsf{L}], \quad U(0) =
(u_{0},0,1,0)^T, \quad U(\mathsf{L}) \in W_{loc}^{s}(\gamma), \quad \text{%
for some } u_{0}, \mathsf{L} \in \mathbb{R},
\end{equation}
can be extended to define an even soliton $u(x) = u_1(|x|)$ of the GP
equation.  To obtain an explicit boundary condition on the stable manifold, we use the Parameterization Method for periodic orbits \cite{param1,param2,param3,Castelli2018}.

The rest of the paper is dedicated to solving the boundary-value problem \eqref{eq:BVP_soliton}.
We reformulate both the construction of a local stable manifold of the periodic orbit \(\gamma\) and the boundary-value problem as zero-finding problems in infinite-dimensional Banach spaces. In both cases, we apply tools from computer-assisted proofs in nonlinear analysis \cite{MR1420838,MR2807595,MR3444942,MR3971222,MR3990999} to prove the existence of solutions near approximate zeros of the corresponding maps. Our approach is based on the Newton–Kantorovich-type theorem stated below.

\begin{theorem}[\textbf{Newton-Kantorovich Theorem}]
\label{Theorem:Newton} Let $X$ and $Z$ be Banach spaces, and let $F: X \to Z 
$ be a $C^{1}$ map, let $\bar{x}$ be an element of $X$, $A$ a linear
injective map from $Z$ to $X$. Let $r^{*}$ be a positive real number and
denote by $B(\bar{x}, r^{*})$ the closed ball centered at $\bar{x}$ of
radius $r^{*}$. Assume there exist nonnegative constant Y, $Z_{1}$ and $%
Z_{2} $ such that 
\begin{align}  \label{eq:Z2_bound_general}
\lVert A F(\bar{x}) \rVert_{X} & \le Y \\
\lVert I - ADF(\bar{x}) \rVert_{B(X)} & \le Z_{1} \\
\lVert A(DF(x) - DF(\bar{x})) \rVert_{B(X)} & \le Z_{2} \lVert x- \bar{x}
\rVert_{X} , \quad \forall x\in B(\bar{x}, r^{*}).
\end{align}
If $Z_{1} < 1$ and $Z_{2} < \frac{(1 - Z_{1})^{2}}{2Y}$, then for any $r$
satisfying $\frac{1 - Z_{1} - \sqrt{ (1- Z_{1})^{2} - 2YZ_{2}}}{Z_{2}} \leq
r < min\left( \frac{1-Z_{1}}{Z_{2}}, r^{*} \right)$, then there exists a
unique zero $x^{*}$ of $F$ in the ball $B(\bar{x},r)$.
\end{theorem}
\add{A proof of this theorem can be found in \cite{vanDenBerg2021}.} In our setting, \(\bar{x}\) denotes a numerical approximate solution. The quantities \(Y\), \(Z_1\), and \(Z_2\) correspond to bounds on the norms of elements in certain infinite-dimensional Banach spaces. For any given set of parameters, we derive these bounds in a form that can be rigorously evaluated by a computer. This derivation constitutes the main technical part in our paper. 

\emph{While other approaches to studying the existence of gap solitons exist, our main contribution is to provide explicit conditions that guarantee the existence of a true solution near a numerical approximation. Moreover, the examples we present constitute the first computer-assisted proofs of soliton existence in the Gross--Pitaevskii equation in non-perturbative parameter regimes.}

Our paper is structured as follows. In Section~\ref{sec:manifold}, we
describe a computational method to obtain a parameterization of the local
stable manifold $W_{loc}^{s}(\gamma) $ with rigorous error bounds. In
Section~\ref{sec:BVP}, we introduce a constructive approach and prove the
existence of solutions to the BVP \eqref{eq:BVP_soliton} using Chebyshev
series expansions. Finally, in Section~\ref{sec:conclusions}, we provide examples of our method, including a computer-assisted proof of Theorem~\ref{theorem:main}.


\section{Computation of a Local Stable Manifold of the Periodic Orbit} \label{sec:manifold}
In Section~\ref{sec:intro}, we reformulated the problem of proving the existence of soliton solutions as finding a solution to equation \eqref{eq:4D_system} that intersects the local stable manifold \( W_{loc}^s(\gamma) \) associated with the periodic orbit \( \gamma \). To explicitly characterize points on this stable manifold, we will employ the Parameterization Method \cite{param1,param2,param3}, following the framework developed in \cite{Castelli2015,Castelli2018}. The periodic orbit $\gamma$ possesses two trivial Floquet exponents: one arising from the conserved quantity $H = \frac{1}{2}(u_{4}^{2} + 4u_{3}^{2})$, discussed in Section~\ref{sec:intro}, and another due to the shift invariance of the periodic orbit. Consequently, $\gamma$ admits at most two nontrivial Floquet exponents. For the remainder of this work, we assume that $\dim W^u(\gamma) = \dim W^s(\gamma) = 1$ and denote the stable Floquet exponent of $\gamma$ by $\lambda < 0$. \add{To compute the stable bundle $v$ associated with the periodic solution $\gamma$, 
we linearize the vector field $g$ in \eqref{eq:4D_system} around $\gamma$, 
yielding the linearized equation
\begin{equation} \label{eq:bundle_ODE_0}
\dot u(\theta) = Dg(\gamma(\theta))\, u(\theta),
\end{equation}
where the matrix $Dg(\gamma(\theta))$ is $\pi$-periodic.  
By Floquet theory, solutions of this linear system can be expressed as
\begin{equation} \label{eq:bundle_ODE_1}
u(\theta) = e^{\lambda \theta} v(\theta),
\end{equation}
with $\lambda \in \mathbb{R}$ the Floquet exponent and $v(\theta)$ a $2\pi$-periodic function.  
The period of $v(\theta)$ is doubled relative to that of the original solution to account 
for the possible non-orientability of the stable bundle. Plugging the expression \eqref{eq:bundle_ODE_1} in \eqref{eq:bundle_ODE_0} yields the functional equation
\begin{equation} \label{eq:bundle_ODE}
    \dot{v} + \lambda v = Dg(\gamma(\theta))v,
\end{equation}
where the solution $v:S^1 \to \mathbb{R}^{4}$ denotes the associated stable bundle.
}
We refer to the image of $v$ as the {\em stable tangent bundle} attached to the periodic orbit \(\gamma\).

Once  the solution $v$ of \eqref{eq:bundle_ODE} is obtained, the Parameterization Method allows us to compute a parameterization \(W: S^1 \times [-1,1] \to \mathbb{R}^{4}\) of \(W_{loc}^s(\gamma)\) by solving the following partial differential equation
\begin{equation} \label{eq:manifold_PDE}
    \frac{\partial }{\partial \theta} W(\theta,\sigma) + \lambda\sigma \frac{\partial }{\partial \sigma} W(\theta,\sigma) = g(W(\theta,\sigma)),
\end{equation}
subject to the following first order constraints
\begin{equation} \label{eq:first_order_constraints}
    W(\theta,0) = \gamma(\theta) \quad \text{and} \quad \frac{\partial }{\partial \sigma} W(\theta,0) = v(\theta).
\end{equation}
Observe that, while the period of \( \gamma \) is \( \pi \), we define the domain of the local stable manifold in \( \theta \) as \( S^1 \bydef \mathbb{R}/(2\pi\mathbb{Z}) \) to accommodate the potential non-orientability of the manifold.

As previously established (e.g., see Theorem 2.6 in \cite{Castelli2015}), if \( W \) satisfies \eqref{eq:manifold_PDE} and \eqref{eq:first_order_constraints}, and \( \varphi \) denotes the flow generated by \( \dot{U} = g(U) \) as given in \eqref{eq:4D_system}, then the following conjugacy relation is satisfied:
\begin{equation} \label{eq:conjugacy_manifold}
    \varphi\left( W(\theta,\sigma),t\right) = W(\theta+t,e^{\lambda t}\sigma)
\end{equation}
for all $\sigma \in [-1,1]$, $\theta \in S^1$ and $t \ge 0$.

To construct a parameterization \(W: S^1 \times [-1,1] \to \mathbb{R}^{4}\) satisfying \eqref{eq:manifold_PDE} and \eqref{eq:first_order_constraints}, we adopt a sequence space framework.
Let us formalize this. In order to represent a sequence of Fourier coefficients of a periodic function, we introduce the following {\em sequence space}
\begin{equation} \label{eq:S_F}
    S_{\sF} \bydef  \left\{ s = (s_m)_{m \in \Z} :  s_m \in \mathbb{C},  \lVert s \rVert_{\sF} \bydef   \sum_{m\in \mathbb{Z}}  |s_m| \nu^{|m|} < \infty \right\},
\end{equation}
for a given exponential weight $\nu \ge 1$.  

Given two sequences of complex numbers $u_1 = \{(u_1)_m\}_{m\in \Z},u_2 = \{(u_2)_m\}_{m\in \Z} \in S_{\sF}$, denote their {\em discrete convolution} given component-wise by
\begin{equation} \label{eq:discrete_convolution}
    (u_1 *_{\sF}u_2)_m \bydef \sum_{m_1+m_2=m \atop m_1,m_2 \in \Z} (u_1)_{m_1} (u_2)_{m_2}.
\end{equation}
This gives rise to a product $*_{\sF}:S_{\sF} \times S_{\sF} \to S_{\sF}$. To represent the Taylor-Fourier coefficients of the parameterization $W$, we consider the sequence space $S_{\sTF}$ defined by
\begin{equation} \label{eq:S_TF}
    S_{\sTF} \bydef  \left\{w = \{w_{n}\}_{n\geq0}: w_n \in S_{\sF}  , \lVert w \rVert_{\sTF}  \bydef \sum_{n\geq0} \lVert w_n \rVert_{\sF} <\infty \right\}.
\end{equation}
Given a Taylor-Fourier sequence $p \in S_{\sTF}$, and given $n \ge 0$, we denote by $p_n \in S_{\sF}$ the Fourier sequence
$p_n \bydef ( p_{n,m})_{m \in \Z}$. Now, given $p,q \in S_{\sTF}$  we define their Taylor-Fourier Cauchy product  $*_{\sTF} :S_{\sTF} \times S_{\sTF} \to S_{\sTF}$ as follows
\begin{equation} \label{eq:Taylor_Fourier_product}
    (p *_{\sTF} q)_n \bydef \sum_{l =0}^n p_l*_{\sF} q_{n-l}.
\end{equation}
Having formalize some sequence spaces in which we will work, we now express \(W\) as a Taylor series in \(\sigma\), with each Taylor coefficient further expanded as a Fourier series in \(\theta\), that is
\begin{equation} \label{eq:TF_expansion}
    W(\theta, \sigma) = \sum_{n = 0}^{\infty} W_n(\theta) \sigma^{n} =
    \sum_{n = 0}^{\infty}
    \sum_{m \in \mathbb{Z}}
    w_{n,m } e^{im\theta} \sigma^{n}, \quad w_{n,m } \in \mathbb{C}^4,
\end{equation}
where the real periodic function $W_n(\theta)$ is expressed as a Fourier series
\begin{equation} \label{eq:Fourier_manifold_expansion}
    W_n(\theta) \bydef \sum_{m \in \mathbb{Z}}    w_{n,m } e^{im\theta}.
\end{equation}
We introduce a notation that will be used throughout this paper: superscripts denote the components of vector sequence variables. For example, a vector \( v \in \mathbb{C}^4 \) is written as \( v = (v^{(1)}, v^{(2)}, v^{(3)}, v^{(4)}) \).
From the constraints \eqref{eq:first_order_constraints}, it follows that $W_0(\theta)=\gamma(\theta)$ and $W_1(\theta)=v(\theta)$. The Fourier series of each component $\gamma^{(j)}$ ($j=1,\dots,4$) of the periodic orbit $\gamma(\theta)$ defined in \eqref{eq:periodic_orbit} is given by
\[
    \gamma^{(j)}(\theta) = \sum_{m \in \mathbb{Z}} \gamma^{(j)}_{m} e^{im\theta},
\]
where \(\gamma^{(1)}_{m} = \gamma^{(2)}_{m} = 0\) for all \(m \in \mathbb{Z}\), while the Fourier coefficients $\gamma^{(3)}$ and $\gamma^{(4)}$ are given by
\begin{equation} \label{eq:gamma34_fourier_coeffs}
    \gamma^{(3)}_{m} =
    \begin{cases}
        \frac{1}{2}, & m = \pm 2,    \\
        0,           & m \neq \pm 2,
    \end{cases}
    \quad \text{and} \quad
    \gamma^{(4)}_{m} =
    \begin{cases}
        -i, & m = -2,       \\
        i,  & m = 2,        \\
        0,  & m \neq \pm 2.
    \end{cases}
\end{equation}
From the Fourier-Taylor expansion \eqref{eq:TF_expansion}, we then get that $w_{0,m}^{(j)} = \gamma_{m}^{(j)}$ for all $m\in \mathbb{Z}$ and $j = 1,2,3,4$. Having derived an explicit expression for the Fourier coefficients of the periodic orbit, we now turn to computing the Fourier coefficients of the stable tangent bundle \( W_1(\theta) = v(\theta) \), which we express as
\begin{equation} \label{eq:v(theta)_Fourier}
    W_1(\theta)=v(\theta) = \sum_{m \in \Z} v_{m} e^{im\theta}, \quad v_{m} = (v_{m}^{(1)},\dots,v_{m}^{(4)}) \in \mathbb{C}^4.
\end{equation}
To compute the coefficients of \( v = W_1 \), we substitute the Fourier expansion \eqref{eq:v(theta)_Fourier} into the linear non-autonomous differential equation \eqref{eq:bundle_ODE}, match terms with like powers, and derive the following infinite system of algebraic equations indexed over $m\in \mathbb{Z}$:
\begin{equation} \label{eq:bundle_coeff}
    \begin{aligned}
        (im+ \lambda)v^{(1)}_{m} & = v^{(2)}_{m}
        \\
        (im+ \lambda)v^{(2)}_{m} & = -a v^{(1)}_{m} + b (\gamma^{(3)}*_{\sF}v^{(1)})_{m}
        \\
        (im+ \lambda)v^{(3)}_{m} & = v_{m}^{(4)}
        \\
        (im+ \lambda)v^{(4)}_{m} & = -4v_{m}^{(3)},
    \end{aligned}
\end{equation}
where $\gamma^{(3)}*_{\sF}v^{(1)}$ denotes the discrete convolution of $\gamma^{(3)}$ and $v^{(1)}$ as introduced in \eqref{eq:discrete_convolution}.
From the sequence equation \eqref{eq:bundle_coeff} follows that
\begin{equation}
    \label{eq:bundle_v3_v4}
    \begin{pmatrix}
        im +\lambda & -1          \\
        4           & im +\lambda
    \end{pmatrix}
    \begin{pmatrix}
        v^{(3)}_{m} \\
        v^{(4)}_{m}
    \end{pmatrix}
    =
    \begin{pmatrix}
        0 \\ 0
    \end{pmatrix}, \quad m\in \mathbb{Z}.
\end{equation}
Since \(\lambda < 0\), the linear system admits a unique solution, which is the zero vector for all \(m \in \mathbb{Z}\). Consequently, \(v_m^{(3)} = v_m^{(4)} = 0\) for all \(m \in \mathbb{Z}\). Thus, it remains to rigorously compute the coefficients of \( v^{(1)} \) and \( v^{(2)} \), which we carry out in Section~\ref{sec:bundle}. \add{After obtaining a proof in the sequence space \(\sF\), we show that the Fourier coefficients \(v_m\) of the solution have the property that the negative modes are the complex conjugates of the corresponding positive modes, so the associated Fourier series represents a real-valued periodic function.} Assuming that this is done, from the Fourier-Taylor expansion \eqref{eq:TF_expansion}, we then get that $w_{1,m}^{(j)} = v_{m}^{(j)}$ for all $m\in \mathbb{Z}$ and $j = 1,2,3,4$.

Having established a strategy for obtaining the Fourier coefficients of \( W_n(\theta) \) for \( n=0,1 \) in \eqref{eq:Fourier_manifold_expansion}, we now proceed to compute the higher-order Taylor coefficients for \( n \ge 2 \). Substituting the Fourier-Taylor expansion \eqref{eq:TF_expansion} into the PDE \eqref{eq:manifold_PDE} and equating terms with like powers results in the following relations, indexed over $m \in \Z$ and $n \ge 2$:
\begin{equation}\label{eq:manifold_coeff}
    \begin{aligned}
        (im + n\lambda ) w_{n,m}^{(1)} & =
        w_{n,m}^{(2)}
        \\
        (im + n\lambda ) w_{n,m}^{(2)} & =
        - a w_{n,m}^{(1)}
        + b \left( w^{(3)}*_{\sTF}w^{(1)} \right)_{n,m}
        +c\left(w^{(1)}*_{\sTF}w^{(1)}*_{\sTF}w^{(1)}\right)_{n,m}
        \\
        (im + n\lambda )w_{n,m}^{(3)}  & =
        w_{n,m}^{(4)}
        \\
        (im + n\lambda ) w_{n,m}^{(4)} & =
        -4w_{n,m}^{(3)},
    \end{aligned}
\end{equation}
where the Taylor-Fourier Cauchy product $*_{\sTF}$ is given in \eqref{eq:Taylor_Fourier_product}.
By applying a similar argument to the one used to establish that \(v^{(3)} = v^{(4)} = 0\), we conclude that \(w_{n,m}^{(3)} = w_{n,m}^{(4)} = 0\) for all \(n \geq 2\) and \(m \in \mathbb{Z}\). Thus, the remaining task in parameterizing \(W_{loc}^s(\gamma)\) reduces to rigorously enclosing the coefficients \(w_{n,m}^{(1)}\) and \(w_{n,m}^{(2)}\) for all \(n \geq 2\) and \(m \in \mathbb{Z}\).
In the remainder of this section, we develop a general computer-assisted framework to prove the existence of solutions to the first two equations of the sequence equations \eqref{eq:bundle_coeff} and \eqref{eq:manifold_coeff} close to approximate solutions. For each problem, we construct a  {\em validation map} \( F \), whose zeros correspond to the desired solutions. We then apply the Newton--Kantorovich Theorem~\ref{Theorem:Newton} to prove the existence of these solutions.

\subsection{Solving for The First Order Coefficients: The Stable Bundle} \label{sec:bundle}
In this section, we present a method to construct a solution to the first two equations of the stable bundle sequence problem \eqref{eq:bundle_coeff}. In particular, the method provides the necessary conditions to verify that a true solution exists close to an approximate solution. Our approach is based on Theorem~\ref{Theorem:Newton}, which requires the definition of relevant Banach spaces, operators, and explicitly computable bounds \( Y \), \( Z_1 \), and \( Z_2 \).

The unknown Floquet exponent \(\lambda\) is treated as part of the solution, which requires expanding the problem to include a space for \(\lambda\). We write the solution as \(x = (\lambda, v)\), where \(\lambda \in \mathbb{C}\) and \(v = \left( v^{(1)}, v^{(2)} \right)\) lies in the sequence space \(S_{\sF}^2\). This leads to the definition of the Banach space \(X_{\sF}\) as the product of the parameter space \(\mathbb{C}\) and the sequence space \(S_{\sF}^2\):
\[
    X_{\sF} \bydef \mathbb{C} \times S^{2}_{\sF} , \quad \| v \|_{S^{2}_{\sF}} \bydef \max \left\{ \|v^{(1)}\|_{\sF}, \|v^{(2)}\|_{\sF} \right\}, \quad \| x \|_{X_{\sF}} \bydef \max \left\{ |\lambda|, \| v \|_{S^{2}_{\sF}} \right\}.
\]
Although \(\lambda\) is expected to be real, we perform computations in the complex space \(X_{\sF}\) because the sequence space components involve complex coefficients, and our numerical methods use complex floating-point vectors. Once the existence of a solution in the complex space is established, we demonstrate that \(\lambda\) is real and satisfies \(\lambda < 0\).
Observe that if \((\lambda, v)\) solves \eqref{eq:bundle_coeff}, then any scalar multiple of \(v\) is also a solution. To guarantee uniqueness, which is essential for the Newton-Kantorovich approach used in Theorem~\ref{Theorem:Newton}, we introduce a phase condition as an additional equation:
\begin{equation}
    \label{eq:phase_conditon}
    \eta(v) - l = 0,
    \quad \text{where} \quad
    \eta(v) \bydef \sum_{|m|\leq M} v_{m}^{(1)},
    \quad \text{and} \quad l \in \mathbb{R}.
\end{equation}
\add{The phase condition sets the first component of \( v(0) \) equal to \( l \).
} A convenient formulation of the phase condition involves defining the sequence \(\mathbf{1}_M\), indexed over \(\mathbb{Z}\), such that \((\mathbf{1}_M)_m = 1\) for \(|m| \leq M\) and \((\mathbf{1}_M)_m = 0\) for \(|m| > M\). This allows the phase condition to be expressed as the dot product \(\eta(v) = \mathbf{1}_M \cdot v^{(1)} \). With this normalization condition on \(v\), we define the {\em validation map} \(F\) by
\begin{equation} \label{eq:bundle_F}
    F(x) \bydef \begin{pmatrix} \eta(v)-l \\ L_{\lambda}v - f(v) \end{pmatrix},
\end{equation}
where
\begin{equation} \label{eq:L_and_f_bundle}
    L_{\lambda} v \bydef \begin{pmatrix} \left( (im+ \lambda) v^{(1)}_{m} \right)_{m \in \Z} \\
        \left( (im+ \lambda) v^{(2)}_{m} \right)_{m \in \Z}
    \end{pmatrix}
    \quad \text{and} \quad
    f(v) \bydef
    \begin{pmatrix}
        v^{(2)} \\
        -a v^{(1)} +  b\left(\gamma^{(3)}*_{\sF}v^{(1)}\right)
    \end{pmatrix}.
\end{equation}
Note that the first two equations of the right-hand side of equation \eqref{eq:bundle_coeff} are represented by the sequence operator $f: S_{\sF}^2 \to S_{\sF}^2$, while the left-hand side is expressed as a linear operator $L_{\lambda}: S_{\sF}^2 \to S_{\sF}^2$. 
Our computer-assisted approach relies on using a finite dimensional approximate solution to $F(x)=0$ that we have obtained through numerical methods. Therefore, interactions between truncated sequences and infinite sequences are essential to our method. To handle sequences effectively, we introduce a {\em truncation operator} as follows. For a sequence $p$ in the sequence space $S_{\sF}$ and a set of indices $R \subset \Z$, we define the truncation operator as
\[
    \left( \proj{R}{F}  p \right)_{m} =
    \begin{cases}
        p_m, & m\in R     \\
        0,   & m\notin R.
    \end{cases}
\]

We adopt the following conventions regarding the action of truncation operator \add{for a set of indices $R$} on elements $(\lambda, v)$ of the product space $X_{\sF}$
\[
    \proj{R}{F} (\lambda,v) \bydef (0,  \proj{R}{F} v),
    \quad
    \proj{R}{F}  v \bydef
    \begin{pmatrix}
        \proj{R}{F}  v^{(1)} \\
        \proj{R}{F} v^{(2)}
    \end{pmatrix}
    \quad \text{and} \quad
    \proj{}{\mathbb{C}} (\lambda,v) \bydef (\lambda,0).
\]
Moreover, to compactly denote sets of indices we define for $P,Q \in \mathbb{N}$,
\[
    [P,Q] \bydef \left\{ m \in \mathbb{Z} : P \leq |m| \leq Q \right\}, \quad
    (P,Q] \bydef \left\{ m \in \mathbb{Z} : P < |m| \leq Q \right\}, \\
\]
\[
    (Q,\infty ) \bydef \left\{ m \in \mathbb{Z} : |m| > Q \right\}.
\]
\add{We define the {\em support} of a sequence \( p \) as the set of indices \( R \) such that \( p_m \neq 0 \) for all \( m \in R \).} Sequences with finite support can be represented as finite-dimensional vectors, which are suitable for computational manipulation. We refer to such sequences as {\em computable sequences}.
In contrast, sequences whose support lies within the interval \( (M, \infty) \) for some \( M \in \mathbb{N} \) are infinite-dimensional and cannot naturally be represented as finite-dimensional vectors. These sequences are referred to as {\em infinite tails}. With this terminology established, we proceed to describe the subsequent steps of our method, introducing the element \(\bar{x}\) and the operator \( A \) required to apply  Theorem~\ref{Theorem:Newton}. Let \(\bar{x} = \bar{x}_{\sF} \bydef (\bar{\lambda}, \bar{v})\) denote an element of \( X_{{\sF}} \) with finite support, which is
\[
    \bar{x}_{\sF} =   \left( \bar{\lambda}, \proj{[0,M]}{F} \bar{v}  \right).
\]
The operator \( A \), central to Theorem~\ref{Theorem:Newton}, serves as a link between the finite-dimensional and infinite-dimensional components of our approach. It is defined as
\begin{equation} \label{eq:bundle_A}
    A \bydef A_f + A_\infty,
    \qquad  \text{with }
    A_\infty \bydef L_{\bar{\lambda}}^{-1} \proj{(M, \infty)}{F}
    \bydef
    \begin{pmatrix} 0 & 0 \\ 0 & L_{\bar{\lambda}}^{-1}
                \proj{(M,\infty)}{F}
    \end{pmatrix}
    ,
\end{equation}
where \( A_f \) outputs sequences with finite support.
Specifically,
\[
    A_f \bydef \left( \proj{}{\mathbb{C}} + \proj{[0, M]}{F} \right)
    \bar{A}(\bar{x}_{\sF})\left( \proj{}{\mathbb{C}} + \proj{[0, M]}{F} \right)
    =
    \begin{pmatrix} 1 & 0 \\ 0 &
                \proj{[0, M]}{F}
                \bar{A}(\bar{x}_{\sF}) \proj{[0, M]}{F}
    \end{pmatrix} ,
\]
where \( \bar{A}(\bar{x}_{\sF}) \) can be represented as a \( 2(2M+1) \times 2(2M + 1) \) matrix. In practice, we use the numerical inverse matrix of
\(
\left( \proj{}{\mathbb{C}} + \proj{[0, M]}{F} \right)
DF(\bar{x}_{\sF})\left( \proj{}{\mathbb{C}} + \proj{[0, M]}{F} \right),
\)
as the matrix \( \bar{A} \). Moreover, we represent the inverse of the sequence operator $L_\lambda$ acting on $ \left( \lambda, v  \right) \in X_{\sF}$ by
\[
    L_{\lambda}^{-1} (\lambda, v ) \bydef \left(  0,  L_{\lambda}^{-1}  v  \right), \quad \text{where} \quad
    L_{\lambda}^{-1} v \bydef \begin{pmatrix} \left( (im+ \lambda)^{-1} v^{(1)}_{m} \right)_{m \in \Z} \\
        \left( (im+ \lambda)^{-1} v^{(2)}_{m} \right)_{m \in \Z}
    \end{pmatrix}.
\]

Since we plan to use computers for our calculations, we need to determine what mathematical objects can be implemented computationally. A mathematical object is called {\em computable} if it can be explicitly implemented in a computational program, meaning it can be constructed or evaluated through a finite sequence of well-defined computational steps. Computable expressions can be rigorously evaluated using {\em interval arithmetic}, where numbers are represented as intervals instead of single floating-point values. Arithmetic operations such as addition, multiplication, and division are performed directly on these intervals. This approach propagates numerical uncertainty throughout the computation, producing intervals that rigorously enclose the true result (e.g., see \cite{MR0231516,MR2807595}).

\add{Even if our notation based on truncation operators may appear cumbersome at first, our choice is deliberate. It is designed so that the bounds and operations in the paper closely match those implemented in the computer code, making the translation from paper to code as transparent as possible and simplifying the verification of our implementation.}

We now turn to deriving computable bounds necessary for Theorem~\ref{Theorem:Newton}. Our strategy is to decompose these bounds into two components: an explicitly computable part and a component involving infinite tails, which will be estimated using computable quantities. To this end, we utilize the following sequence relations.
\begin{lemma}
    \label{lemma:convolution_product}
    Let $p$, $q$ be sequences supported in $[0,M]$ and let $r$ be a sequence supported in \((2M, \infty)\). Then,
    \begin{equation} \label{eq:convolution_finite_finite}
        \proj{[0,M]}{F} p *_{\sF} \proj{[0,M]}{F} q = \proj{[0,2M]}{F} \left( \proj{[0,M]}{F} p *_{\sF} \proj{[0,M]}{F} q \right),
    \end{equation}
    that is the product of two finite sequences remains finite, though it extends the support of the resulting sequence. Moreover, the convolution of a sequence supported in \([0, M]\) and a sequence supported in \((2M, \infty)\), results in a sequence supported in \((M, \infty)\), that is
    \begin{equation} \label{eq:convolution_finite_infinite}
        \proj{[0, M]}{F} p *_{\sF} \proj{(2M, \infty)}{F} r = \proj{(M, \infty)}{F}  \left( \proj{[0, M]}{F} p *_{\sF} \proj{(2M, \infty)}{F} r  \right)
    \end{equation}
\end{lemma}

\begin{proof}
    The proof of \eqref{eq:convolution_finite_finite} is straightforward, as products of $M^{th}$ order trigonometric polynomials yield a $2M^{th}$ order trigonometric polynomial. For the proof of \eqref{eq:convolution_finite_infinite}, suppose that $0 < k \leq M$. By definition,
    \[
        \left( p *_{\sF} r \right)_{k} = \sum_{m \in \Z} p_{k - m} r_{m}.
    \]
    For all $m \in [0, 2M]$, $r_{m} = 0$. Now, for $m > 2M$, we have $k - m \leq -M$. If $m < -2M$, then $k-m>3M$. Therefore for $m \in (2M, \infty)$ we have  $p_{k - m} = 0$. The argument is analogous when $-M < k \leq 0$. Therefore we have that $\left( p *_{\sF} r \right)_{k} = 0$ for all $k \in [0, M]$.
\end{proof}

We are now prepared to establish an explicit, computable bound \( Y \) that satisfies \eqref{eq:Z2_bound_general}.

\begin{lemma} \label{lemma:bundle_Y}
    A computable upper bound for $\lVert A F (\bar{x}_{\sF}) \rVert_{X_{\sF}}$ is given by
    \begin{equation} \label{eq:bundle_Y}
        Y(\bar{x}_{\sF}) \bydef \lVert A_{f} F(\bar{x}_{\sF})\rVert_{X_{\sF}}
        +  \lVert  \proj{(M,2M]}{F}   L_{\bar{\lambda}}^{-1}  F(\bar{x}_{\sF})\rVert_{X_{\sF}}.
    \end{equation}
\end{lemma}
\begin{proof}
    By construction, $A F (\bar{x}_{\sF})=A_{f} F(\bar{x}_{\sF}) + A_{\infty} F (\bar{x}_{\sF})$. Since $f(\bv)$ involves the product of two finite sequences, each supported in $[0, M]$, equation  \eqref{eq:convolution_finite_finite} from Lemma~\ref{lemma:convolution_product} implies that $F(\bar{x}_{\sF})$ involves sequences with support in $[0, 2M]$. Now, since \( A_f \) maps inputs with finite support to outputs with finite support, the quantity $\lVert A_{f} F(\bar{x}_{\sF})\rVert_{X_{\sF}}$ is computable. Furthermore, once again using Lemma~\ref{lemma:convolution_product}, we have
    \[
        A_{\infty}  F (\bar{x}_{\sF})
        =
        L_{\bar{\lambda}}^{-1}  \proj{(M,\infty)}{F}  F(\bar{x}_{\sF})
        =
        L_{\bar{\lambda}}^{-1}  \proj{(M,\infty)}{F}  \proj{[0,2M]}{F}   F(\bar{x}_{\sF})
        =
        \proj{(M,2M]}{F} L_{\bar{\lambda}}^{-1} F(\bar{x}_{\sF}).
    \]
    We used the fact that the composition of two truncation operators is supported on the intersection of their supports. Consequently a computable bound for $\lVert A F (\bar{x}_{\sF}) \rVert_{X_{\sF}}$ is given by \eqref{eq:bundle_Y}.
\end{proof}
We now proceed to derive the bound \( Z_1 = Z_1(\bar{x}_{\sF}) \) that satisfies \eqref{eq:Z2_bound_general}. \add{In what follows, we write \( B(X_{\sF}) \) for the space of continuous linear operators on \( X_{\sF} \).}

%

\begin{lemma} \label{lemma:bundle_Z1}
    A computable upper bound for $\lVert  I - A D F (\bar{x}_{\sF})  \rVert_{B(X_{\sF})}$ is given by
    \begin{equation} \label{eq:bundle_Z1}
        Z_1(\bar{x}_{\sF}) \bydef \lVert \proj{}{\mathbb{C}}  + \proj{[0,M]}{F} - A_{f} D F (\bar{x}_{\sF}) \left(\proj{}{\mathbb{C}}  + \proj{[0,2M]}{F}\right) \rVert_{B(X_{\sF})} +
        \frac{\max \left\{ 1, |a| + | b | \nu^2 \right \}}{\sqrt{(M+1)^2 + {\bar \lambda}^2}}  .
    \end{equation}
\end{lemma}
\begin{proof}
    We begin by decomposing the action of the operator $ADF(\bar{x}_{\sF})$ into its finite and infinite parts.
    \begin{align}
        \nonumber
        \lVert  I - A D F (\bar{x}_{\sF})  \rVert_{B(X_{\sF})}
         & =
        \lVert \proj{}{\mathbb{C}} + \proj{[0, M]}{F} + \proj{(M,2M]}{F} + \proj{(2M,\infty)}{F} - \left(A_{f} + A_{\infty}\right) D F (\bar{x}_{\sF}) \left(\proj{}{\mathbb{C}}  + \proj{[0,2M]}{F}  + \proj{(2M,\infty)}{F} \right) \rVert_{B(X_{\sF})}
        \\
        \label{eq:Z1_splitting}
         & \leq
        \lVert  \proj{}{\mathbb{C}} + \proj{[0, M]}{F}  - A_{f} D F (\bar{x}_{\sF}) \left(\proj{}{\mathbb{C}}  + \proj{[0,2M]}{F}\right) \rVert_{B(X_{\sF})} +
        \lVert A_{f} D F (\bar{x}_{\sF})  \proj{(2M,\infty)}{F} \rVert_{B(X_{\sF})}
        \\  & +
        \lVert  \proj{(M, 2M]}{F} -A_{\infty} D F (\bar{x}_{\sF}) \left(\proj{}{\mathbb{C}}  + \proj{[0,2M]}{F}\right)
        +
        \proj{(2M,\infty)}{F} - A_{\infty} D F (\bar{x}_{\sF}) \proj{(2M, \infty)}{F}\rVert_{B(X_{\sF})}.
        \nonumber
    \end{align}

    Observe that, we can express the sequence operator $f$  as a multiplicative linear operator as follows, for $v=(v^{(1)},v^{(2)}) \in S_{\sF}^2$:
    \begin{equation} \label{eq:M_f_bundle}
        f(v) = f v \bydef
        \begin{pmatrix} 0 & I \\ - a I + b(\gamma^{(3)}*_{\sF} \cdot) & 0
        \end{pmatrix}v
        =
        \begin{pmatrix}
            v^{(2)} \\
            -a v^{(1)} + b(\gamma^{(3)}*_{\sF} v^{(1)})
        \end{pmatrix},
    \end{equation}
    where we identify the linear map $f$ with its associated multiplication operator.
    Moreover, recalling the definition of the map $F$ in \eqref{eq:bundle_F}, the Fréchet derivative of the validation map is given by
    \begin{equation} \label{eq:DF(bar_x)}
        DF(\bar{x}_{\sF}) = \begin{pmatrix} 0 & \mathbf{1}_M^T \\ \bv & L_{\bar{\lambda}} - f \end{pmatrix}.
    \end{equation}
    %
    At this point, it is crucial to emphasize a key property of the operator $DF(\bar{x}_{\sF})$: it has finite bandwidth. This follows from the fact that $\mathbf{1}_M$ has finitely many nonzero entries, and that the linear multiplication operator $f$ defined in \eqref{eq:M_f_bundle} has finite bandwidth due to the finite number of nonzero Fourier coefficients of \(\gamma^{(3)}\) (specifically two, as noted in \eqref{eq:gamma34_fourier_coeffs}). With this property established, we proceed to derive a computable upper bound for each of the three terms on the right-hand side of the inequality in \eqref{eq:Z1_splitting}. First, since $DF(\bar{x}_{\sF})$ has finite bandwidth, the first term in \eqref{eq:Z1_splitting} is naturally computable.

    To address the second term, note that in \eqref{eq:Z1_splitting} we choose a finite truncation in \([0,2M]\) because \(f\) defined in \eqref{eq:L_and_f_bundle} involves only linear terms and the convolution with a sequence truncated up to mode \(M\). Therefore, to evaluate \(A_f\) at the infinite tail of a sequence we apply \eqref{eq:convolution_finite_infinite} from Lemma~\ref{lemma:convolution_product} and we obtain
    \begin{equation}
        \begin{aligned}
            \nonumber
            A_{f} D F (\bar{x}_{\sF}) \proj{(2M,\infty)}{F}
             & =
            \begin{pmatrix} 0 & 0 \\ 0 &
                \proj{[0, M]}{F}
                \bar{A}(\bar{x}_{\sF}) \proj{[0, M]}{F}
            \end{pmatrix}
            \begin{pmatrix} 0 & \mathbf{1}_M^T \\ \bv & L_{\bar{\lambda}} - f \end{pmatrix}
            \begin{pmatrix} 0 & 0 \\ 0 & \proj{(2M,\infty)}{F}
            \end{pmatrix}
            \\ &=
            \begin{pmatrix} 0 & 0 \\ 0 &
                \proj{[0, M]}{F}
                \bar{A}(\bar{x}_{\sF})
            \end{pmatrix}
            \begin{pmatrix} 0 & 0 \\ 0 & \proj{[0, M]}{F}
            \end{pmatrix}
            \begin{pmatrix} 0 & \mathbf{1}_M^T \\ \bv & L_{\bar{\lambda}} - f \end{pmatrix}
            \begin{pmatrix} 0 & 0 \\ 0 & \proj{(2M,\infty)}{F}
            \end{pmatrix}
            \\ &=
            \begin{pmatrix} 0 & 0 \\ 0 &
                \proj{[0, M]}{F}
                \bar{A}(\bar{x}_{\sF})
            \end{pmatrix}
            \begin{pmatrix} 0 & 0 \\ 0 & \ -\proj{[0,M]}{F}  f  \proj{(2M,\infty)}{F}
            \end{pmatrix}
            \\ &=
            \begin{pmatrix} 0 & 0 \\ 0 &     \proj{[0, M]}{F}
                \bar{A}(\bar{x}_{\sF})
            \end{pmatrix}
            \begin{pmatrix} 0 & 0 \\ 0 & \ -\proj{[0,M]}{F} \proj{(M,\infty)}{F}  f  \proj{(2M,\infty)}{F}
            \end{pmatrix} = \begin{pmatrix} 0 & 0 \\ 0 & 0
                            \end{pmatrix}.
        \end{aligned}
    \end{equation}
    Above, we first utilized the fact that \( L_{\bar{\lambda}} \) is a diagonal operator, which implies \( \proj{[0,M]}{F} L_{\bar{\lambda}} \proj{(2M,\infty)}{F} = 0 \). Next, we employed the relation \( f \proj{(2M,\infty)}{F} = \proj{(M,\infty)}{F} f \proj{(2M,\infty)}{F} \), followed by the observation that
    $\proj{[0,M]}{F} f \proj{(2M,\infty)}{F} = \proj{[0,M]}{F} \proj{(M,\infty)}{F} f \proj{(2M,\infty)}{F} = 0$, as it is evident that \( \proj{[0,M]}{F} \proj{(M,\infty)}{F} = 0 \). To handle the third term in \eqref{eq:Z1_splitting}, note that
    \begin{equation}
        \begin{aligned}
            \nonumber
            \proj{(M, 2M]}{F}-A_{\infty} D F (\bar{x}_{\sF}) \left(\proj{}{\mathbb{C}}  + \proj{[0,2M]}{F}\right)
             & =
            \begin{pmatrix} 0 & 0 \\ 0 & \proj{(M,2M]}{F}
            \end{pmatrix}
            -
            \begin{pmatrix} 0 & 0 \\ 0 & L_{\bar{\lambda}}^{-1}
                \proj{(M,\infty)}{F}
            \end{pmatrix}
            \begin{pmatrix} 0 & \mathbf{1}_M^T \\ \bv & L_{\bar{\lambda}} - f \end{pmatrix}
            \begin{pmatrix} 1 & 0 \\ 0 & \proj{[0,2M]}{F}
            \end{pmatrix}
            \\ &=
            \begin{pmatrix} 0 & 0 \\ 0 & \proj{(M,2M]}{F}
            \end{pmatrix}-
            \begin{pmatrix} 0 & 0 \\ 0 &L_{\bar{\lambda}}^{-1}  \proj{(M,2M]}{F}L_{\bar{\lambda}} \proj{[0,2M]}{F}
            \end{pmatrix}
            +
            \begin{pmatrix} 0 & 0 \\ 0 & \  L_{\bar{\lambda}}^{-1}
                \proj{(M,\infty)}{F} f\proj{[0,2M]}{F}
            \end{pmatrix}
            \\ &=
            \begin{pmatrix} 0 & 0 \\ 0 & \proj{(M,2M]}{F}
            \end{pmatrix}-
            \begin{pmatrix} 0 & 0 \\ 0 & \proj{(M,2M]}{F}
            \end{pmatrix}
            +
            \begin{pmatrix} 0 & 0 \\ 0 & \  L_{\bar{\lambda}}^{-1}
                \proj{(M,\infty)}{F} f\proj{[0,2M]}{F}
            \end{pmatrix}
            \\ &=
            \begin{pmatrix} 0 & 0 \\ 0 & \  L_{\bar{\lambda}}^{-1}
                \proj{(M,\infty)}{F} f\proj{[0,2M]}{F}
            \end{pmatrix}
            .
        \end{aligned}
    \end{equation}

    For the term involving the evaluation of the operator \( A_{\infty} \) on the infinite tail
    \begin{equation}
        \begin{aligned}
            \nonumber
            \proj{(2M,\infty)}{F} - \proj{(M, 2M]}{F} A_{\infty} D F (\bar{x}_{\sF}) \proj{(2M,\infty)}{F}
             & =
            \begin{pmatrix} 0 & 0 \\ 0 & \proj{(2M,\infty)}{F}
            \end{pmatrix}
            -
            \begin{pmatrix} 0 & 0 \\ 0 & L_{\bar{\lambda}}^{-1}
                \proj{(M,\infty)}{F}
            \end{pmatrix}
            \begin{pmatrix} 0 & \mathbf{1}_M^T \\ \bv & L_{\bar{\lambda}} - f \end{pmatrix}
            \begin{pmatrix} 0 & 0 \\ 0 & \proj{(2M,\infty)}{F}
            \end{pmatrix}
            \\ &=
            \begin{pmatrix} 0 & 0 \\ 0 & \  L_{\bar{\lambda}}^{-1}
                \proj{(M,\infty)}{F} f \proj{(2M,\infty)}{F}
            \end{pmatrix}
            .
        \end{aligned}
    \end{equation}

    Therefore, the right-hand size of the inequality in \eqref{eq:Z1_splitting} can be bounded by
    \begin{equation*}
        \begin{aligned}
            \lVert  I - A D F (\bar{x}_{\sF})  \rVert_{B(X_{\sF})}
             & \leq
            \lVert \proj{}{\mathbb{C}}  + \proj{[0,M]}{F} - A_{f} D F (\bar{x}_{\sF}) \left(\proj{}{\mathbb{C}}  + \proj{[0,2M]}{F}\right) \rVert_{B(X_{\sF})}
            +
            \lVert  L_{\bar{\lambda}}^{-1}
            \proj{(M,\infty)}{F} f\proj{[0,\infty)}{F}
            \rVert_{B(X_{\sF})}
            \\&
            \leq
            \lVert \proj{}{\mathbb{C}}  + \proj{[0,M]}{F} - A_{f} D F (\bar{x}_{\sF}) \left(\proj{}{\mathbb{C}}  + \proj{[0,2M]}{F}\right) \rVert_{B(X_{\sF})}
            +
            \lVert  L_{\bar{\lambda}}^{-1}
            \proj{(m,\infty)}{F} \rVert_{B(X_{\sF})} \lVert f  \rVert_{B(S_{\sF}^2)}.
        \end{aligned}
    \end{equation*}
    As previously mentioned, the first term is computable. For the second term, note that the product of sequences preserves algebraic properties such as commutativity, associativity, and distributivity over sequence addition. Discrete convolutions also satisfy the following Banach algebra property:
    \begin{equation}
        \label{eq:convolution_property}
        \| p *_{\sF} q \|_{{\sF}} \leq \| p \|_{{\sF}} \| q \|_{{\sF}} \quad \text{for all} \quad p,q \in S_{\sF}.
    \end{equation}
    This allows us to bound the norm of the operator $f$ as follows. Take $\lVert v  \lVert _{\sF} \leq 1 $ we have
    \begin{equation} \label{eq:bundle_f_bound}
        \begin{aligned}
            \lVert f(v ) \rVert_{B( S_{\sF}^2)}
            =
            \lVert
            \begin{pmatrix}
                0                    & I \\
                -aI  + b\gamma^{(3)} & 0
            \end{pmatrix}
            v
            \rVert_{ B(S_{\sF}^2)}
            =
             &
            \max \left\{ \lVert v^{(2)} \rVert_{\sF}, \lVert -av^{(1)} +  b\left(\gamma^{(3)}*_{\sF}v^{(1)}\right) \rVert_{\sF} \right\}
            \\  & \leq
            \max \left\{ \lVert v^{(2)} \rVert_{\sF}, \lVert av^{(1)} \rVert_{\sF} + | b | \lVert \gamma^{(3)}*_{\sF}v^{(1)} \rVert_{\sF} \right\}
            \\  & \leq
            \max \left\{ 1, |a| + | b | \lVert \gamma^{(3)} \rVert_{\sF} \right \}.
        \end{aligned}
    \end{equation}
    Since we know periodic orbit $\gamma$ explicitly, we have that
    \[
        \gamma^{(3)}_{m} =
        \begin{cases}
            \frac{1}{2}, & m = \pm 2,    \\
            0,           & m \neq \pm 2,
        \end{cases}
    \]
    and hence $\lVert \gamma^{(3)}  \rVert_{\sF}=\nu^2$. Next we estimate the norm of the infinite tail of operator $L_{\bar \lambda}^{-1}$ as follows
    \begin{align}
        \label{eq:bundle_L_bound}
        \lVert L_{\bar{\lambda}}^{-1}
        \proj{(M,\infty)}{F} \rVert_{B(X_{\sF})}
         & \leq
        \frac{1}{\sqrt{(M+1)^2 + \bar{\lambda}^2}}.
    \end{align}

    And thus using \eqref{eq:bundle_f_bound} and \eqref{eq:bundle_L_bound} we obtain a computable $Z_1$ bound given by \eqref{eq:bundle_Z1}.
\end{proof}

Note that equation \eqref{eq:bundle_L_bound} provides a straightforward way to obtain a computable upper bound for the operator norm of \(A\):
\begin{equation}
    \label{eq:bundle_A_norm}
    \|A\|_{X_{\sF}} \le \max\left\{\lVert A_{f}  \rVert_{B(X_{\sF})} ~,~ \frac{1}{\sqrt{(M+1)^2 + \bar{\lambda}^2}} \right\}.
\end{equation}
We now compute a computable bound for \(Z_2(\bar{x}_{\sF})\) using this operator norm bound.
\begin{lemma}
    \label{lemma:bundle_Z2}
    A computable $Z_2(\bar{x}_{\sF})$  satisfying  $\lVert A(DF(x) - DF(\bar{x}_{\sF})) \rVert_{B(X_{\sF})} \leq Z_2(\bar{x}_{\sF}) r$  for all $x$ in $B(\bar{x}_{\sF}, r)$ is given by
    \begin{equation}
        \label{eq:bundle_Z2}
        Z_2(\bar{x}_{\sF}) \bydef  2 \max\left\{\lVert A_{f}  \rVert_{B(X_{\sF})} ~,~ \frac{1}{\sqrt{(M+1)^2 + \bar{\lambda}^2}} \right\}.
    \end{equation}
\end{lemma}
\begin{proof}
    Let $x = (\lambda, v)$ be and element in  $B(\bar{x}_{\sF}, r)$.
    There exists $\delta = (\lambda_{\delta}, v_\delta) \in X_{\sF}  $ with $\lVert \delta \lVert_{X_{\sF}} \ \leq r $ such that $x =\bar{x}_{\sF} + \delta$. For any $h = (\lambda, v)$ such that $\lVert h  \lVert_{X_{\sF}} \leq 1 $ we have
    \[
        \| [DF(\bar{x}_{\sF} + \delta) - DF(\bar{x}_{\sF})] h \|_{X_{\sF}}=
        \| (0, \lambda v_{\delta}  + \lambda_{\delta}v) \|_{X_{\sF}}
        \leq
        2r
    \]
    Using \eqref{eq:bundle_A_norm}  we get a computable $Z_2(\bar{x}_{\sF})$ bound as given by \eqref{eq:bundle_Z2}.
\end{proof}
The following result provides the conditions for the existence of a true solution to the ODE bundle sequence equation near a given approximate solution.
\begin{theorem} \label{Theorem:Bundle}
    Fix parameters \(a\), \(b\), and \(c\), a weight \(\nu\) for the norm in \eqref{eq:S_F}, and a scaling factor \(l\) for the phase condition \eqref{eq:phase_conditon}. Let \(\bar{x}_{\sF}= (\bar{\lambda}, \bar{v}) \in X_{\sF}\) such that each component of $\bar{v}$ is supported in $[0,M]$ for some truncation mode \(M\) and that \(\bar{\lambda}<0\) and is real. Assume that the coefficients of \(\bar{v}\) satisfy the symmetry condition:
    \begin{equation} \label{eq:conjugate_assumption_bundle}
        \bar{v}_k^{(i)} = [\bar{v}_{-k}^{(i)}]^{*}, \quad \text{for} \quad i=1,2, \quad k \in \mathbb{Z}.
    \end{equation}
    Additionally, assume that the matrix \(\bar{A}(\bar{x}_{\sF})\) from operator \eqref{eq:bundle_A} is computed as the numerical inverse of
    \[
        \left( \proj{}{\mathbb{C}} + \proj{[0, M]}{F} \right)
        DF(\bar{x}_{\sF})\left( \proj{}{\mathbb{C}} + \proj{[0, M]}{F} \right).
    \]
    Suppose the computable bounds \(Y(\bar{x}_{\sF})\), \(Z_1(\bar{x}_{\sF})\), and \(Z_2(\bar{x}_{\sF})\) defined in \eqref{eq:bundle_Y}, \eqref{eq:bundle_Z1}, and \eqref{eq:bundle_Z2} satisfy
    \[
        Z_{1}(\bar{x}_{\sF}) < 1 \quad \text{and} \quad Z_{2}(\bar{x}_{\sF}) < \frac{(1 - Z_{1}(\bar{x}_{\sF}))^{2}}{2Y(\bar{x}_{\sF})}.
    \]
    Then, the validation map \eqref{eq:bundle_F} has an unique zero
    \(
    x_{\sF} \bydef (\lambda, v^{(1)}, v^{(2)}),
    \)
    in the ball $B_{X_{\sF}}(\bar{x}_{\sF}, r_{\sF})$ where the radius of the ball is given by:
    \[
        r_{\sF}=\frac{1 - Z_1(\bar{x}_{\sF})- \sqrt{ (1- Z_1(\bar{x}_{\sF}))^2 - 2Y(\bar{x}_{\sF})Z_2(\bar{x}_{\sF})}}{Z_2(\bar{x}_{\sF})}.
    \]
    Then the eigenvalue $\lambda$ is real. Moreover, $(\lambda, v^{(1)}, v^{(2)}, 0, 0) \in \mathbb{R} \times S^{4}_{\sF}$, which satisfies the sequence equation \eqref{eq:bundle_coeff}, corresponds to a real solution of \eqref{eq:bundle_ODE}. Finally, if \(\lambda < 0\), this solution provides a parameterization of a stable linear bundle associated with the periodic orbit \(\gamma\).
\end{theorem}

\begin{proof}
    Using Lemmas~\ref{lemma:bundle_Y}, \ref{lemma:bundle_Z1}, and \ref{lemma:bundle_Z2}, it follows from the Newton-Kantorovich Theorem~\ref{Theorem:Newton} that there exists a unique $x_{\sF} =(\lambda, v^{(1)},v^{(2)}) \in X_{\sF} $ in the ball $B_{X_{\sF}}(\bar{x}_{\sF}, r_{\sF})$ with radius $r_{\sF}$ such that $F(x_{\sF}) =0$. We proceed by proving that $\lambda$ is real. Define the conjugation maps $C^{*}: X_{\sF} \to X_{\sF}$ and $C^{*}: S_{\sF} \to S_{\sF}$ as follows
    \[
        C^{*}(x) \bydef \left(\lambda^{*}, C^{*} v^{(1)}, C^{*} v^{(2)}\right), \quad [C^{*}v]_m \bydef v_{-m}^{*}
        \quad \text{for} \quad m \in \mathbb{Z},
    \]
    where given a complex number $z \in \C$, $z^*$ denotes its complex conjugate.
    For any $x \in X_{\sF}$, we have
    \[
        C^{*} F(x) = F(C^{*}(x)).
    \]
    In particular, $C^{*}(x_{\sF})$ is also a zero of $F$. By assumption \eqref{eq:conjugate_assumption_bundle} and since $\bar \lambda$ is real, then $C^{*}(\bar{x}_{\sF})=\bar{x}_{\sF}$. Now, since $\|C^{*}\|_{B(X_{\sF})} \leq 1$
    \[
        \| C^{*}(x_{\sF}) - \bar{x}_{\sF}\|_{X_{\sF}} = \| C^{*}(x_{\sF}) - C^{*}(\bar{x}_{\sF}) \|_{X_{\sF}} =
        \| C^{*}(x_{\sF}-\bar{x}_{\sF}) \| \le \|C^{*}\|_{B(X_{\sF})} \| x_{\sF} - \bar{x}_{\sF} \|_{X_{\sF}}
        \le r,
    \]
    that is $C^{*}(x_{\sF})$ is in the ball $B_{X_{\sF}}(\bar{x}_{\sF}, r_{\sF})$. By uniqueness of the zero, we obtain
    \[
        C^{*}(x_{\sF}) = x_{\sF},
    \]
    which implies that $\lambda$ is real and that the solution $(\lambda, v^{(1)}, v^{(2)}, 0, 0) $ of \eqref{eq:bundle_ODE} coming from $x_{\sF}$ is real.
\end{proof}
In the next section, we present analogous results for the sequence equation associated with the PDE \eqref{eq:manifold_PDE}. Using a solution to \eqref{eq:bundle_coeff} as obtained from Theorem \ref{Theorem:Bundle}, we compute a parameterization of the stable manifold attached to the periodic orbit $\gamma$.

\subsection{Solving for Higher Order Coefficients}

After developing a strategy to solve the bundle equation \eqref{eq:bundle_coeff}, we now focus on solving the Taylor-Fourier sequence equation \eqref{eq:manifold_coeff} to obtain a parameterization of the stable manifold associated with the periodic orbit \(\gamma\). This involves expanding the solution of \eqref{eq:manifold_PDE} as a Taylor series, with each Taylor coefficient further expanded into a Fourier series. We adopt the same method described in the previous section, working in the space of solution coefficients. As already noted at the beginning of Section~\ref{sec:manifold}, we only need to determine the Taylor coefficients for \(w^{(1)}\) and \(w^{(2)}\). The remaining coefficients are given by:
\begin{equation}
    \label{eq:manifold_3_4}
    w_{0,m}^{(j)} \bydef \gamma_{m}^{(j)}   \quad \text{and} \quad w_{1,m}^{(j)} \bydef 0  \quad \text{for} \quad j = 3,4,
\end{equation}
\[
    w_{n,m}^{(3)} = w_{n,m}^{(4)} \bydef 0 \quad \text{for} \quad n\geq 2, \quad m \in \mathbb{Z}.
\]
The Banach space used for the manifold problem is defined as follows:
\[
    X_{\sTF} \bydef S_{\sTF}^2,
    \quad
    \lVert w \rVert_{ X_{\sTF} }= \max \{ \lVert w^{(1)} \rVert_{\sTF} , \lVert w^{(2)} \rVert_{\sTF} \}.
\]
Given a sequence \(s\) in $S_{\sTF}$ and a set of indices \(R\), we define the truncation operators as follows
\[
    \left( \proj{R}{T}   w \right)_{n,m} =
    \begin{cases}        w_{n,m}, & n \in R, m\in\mathbb{Z} \\
                     0,       & \text{otherwise},
    \end{cases}
    \quad \text{and} \quad
    \left( \proj{R}{F}  w \right)_{n,m} =
    \begin{cases}
        w_{n,m}, & n \in \mathbb{N}, m\in R \\
        0,       & \text{otherwise}.
    \end{cases}
\]
As in the previous section, truncation operators act entrywise in the space \(X_{\sTF}\). Note that, since Taylor sequences are one-sided, the ranges for Taylor truncations are indexed over the natural numbers.

We now define a validation map whose zeros correspond to solutions of \eqref{eq:manifold_coeff}. In this case, the definition of the map \(F\) involves the solution of the bundle sequence equation \eqref{eq:bundle_coeff}. To address this, we consider
\[
    x_{\sF} \bydef (\lambda, v^{(1)}, v^{(2)}),
\]
a zero of \eqref{eq:bundle_F} (for a fixed set of parameters $a$, $b$, $c$) that lies in the ball \(B_{X_{\sF}}(\bar{x}_{\sF}, r_{\sF})\), where \(\bar{x}_{\sF} = (\bar{\lambda}, \bar{v}^{(1)}, \bar{v}^{(2)})\) is an element of \(X_{\sF}\) supported in \([0, M]\). Theorem \ref{Theorem:Bundle} provides the existence of such a solution in this form. Now, given \(w =(w^{(1)}, w^{(2)}) \in S_{\sTF}^2 \) and a solution \((\lambda, v) \in X_{\sF}\) to the sequence equation \eqref{eq:bundle_coeff}, we define
\begin{equation} \label{eq:manifold_F}
    F(w) \bydef \proj{[0,1]}{T} B(w,v) + \proj{[2,\infty)}{T}\left[L_{\lambda}w - f(w) \right],
\end{equation}
where \(B: S_{\sTF}^2 \times S_{\sF}^2 \to S_{\sTF}^2\) is given by $B(w,v)=\left(B^{(1)}(w,v),B^{(2)}(w,v) \right)$, with
\begin{equation}
    \label{eq:manifold_F_B}
    [B^{(i)}(w,v)]_{n,m} \bydef
    \begin{cases}
        w_{0,m}^{(i)} - \gamma^{(i)}_m, & n=0, m \in \mathbb{Z}      \\
        w_{1,m}^{(i)} - v^{(i)}_m,      & n=1, m \in \mathbb{Z}      \\
        0,                              & n\geq 2, m \in \mathbb{Z},
    \end{cases}
\end{equation}
where the linear operator $L_{\lambda}: S_{\sTF}^2 \to S_{\sTF}^2$ is defined for $i=1,2$ as
\[
    \left[L_{\lambda} w\right]_{n,m}^{(i)} \bydef
    \begin{cases}
        w_{n,m}^{(i)},                             & n = 0,1           \\
        \left(im + n\lambda \right) w_{n,m}^{(i)}, & \text{otherwise},
    \end{cases}
\]
and where
\[
    f(w)= f(w^{(1)}, w^{(2)} ) \bydef
    \begin{pmatrix}
        w^{(2)} \\
        -aw^{(1)} +  b\left(w^{(3)}*_{\sTF}w^{(1)}\right) +  c\left(w^{(1)}*_{\sTF}w^{(1)}*_{\sTF}w^{(1)}\right)
    \end{pmatrix}.
\]

We use the same notation for the validation map \(F\), the vector field \(f\) and the linear operator $L$ as in the bundle problem from the previous section. However, to distinguish between the maps, we use different notation for the variables, making it clear from the evaluation which function is being referenced. We now choose \(N\) as the Taylor truncation mode and proceed to define the operator \(A\). Note that the infinite tail of a sequence in \(S_{\sTF}\) is given by the following truncation
\[
    \proj{[0,N]}{T}  \proj{(M,\infty)}{F}
    +
    \proj{(N,\infty)}{T}.
\]
Observe that the definition of $F$ in \eqref{eq:manifold_F} is different in the first two Taylor coefficients. Thus, we define
\begin{equation}
    \label{eq:manifold_A}
    A \bydef A_{f} + A_{\infty},
    \quad
    A_{\infty} \bydef
    \proj{[0,1]}{T}  \proj{(M,\infty)}{ F}
    +
    \proj{[2,N]}{T}  \proj{(M,\infty)}{ F}L_{\lambda}^{-1} +
    \proj{(N,\infty)}{T} L_{\lambda}^{-1},
\end{equation}
where
\[
    \left[L_{\lambda}^{-1} w\right]_{n,m}^{(i)} =
    \begin{cases}
        w_{n,m}^{(i)},                                    & n = 0,1           \\
        \left( im + n\lambda \right)^{-1} w _{n,m}^{(i)}, & \text{otherwise.}
    \end{cases}
\]
and with
\[
    A_{f} \bydef \proj{[0,N]}{T}   \proj{[0,M]}{F} \bar{A}(\bar{w}) \proj{[0,N]}{T}   \proj{[0,M]}{F},
\]
where \(\bar{A} = \bar{A}(\bar{w})\) corresponds to a finite-dimensional matrix.
In practice, we use the numerical inverse of the truncated derivative of the map \(F\) at \(\bar{w}\). Note that we suppose that $\bar{w} \in S_{\sTF}^2$ satisfies
\begin{equation*}
    \bar{w}  = \proj{[0,N]}{T} \proj{[0,M]}{F} \bar{w} =  \left( \proj{[0,N]}{T} \proj{[0,M]}{F} \bar{w}^{(1)},  \proj{[0,N]}{T} \proj{[0,M]}{F} \bar{w}^{(2)}  \right),
\end{equation*}
and we define the first two Taylor coefficients of $\bar{w}_{n,m}^{(1)}$ and $\bar{w}_{n,m}^{(2)}$ as follows
\begin{equation} \label{eq:w_bar_0_1}
    \bar{w}_{0,m}^{(j)} = 0 \quad \text{and} \quad \bar{w}_{1,m}^{(j)} = \bar{v}_{m}^{(j)}  \quad \text{for} \quad j = 1,2.
\end{equation}
We now present a set of bounds, as required by Theorem~\ref{Theorem:Newton}, that can be explicitly computed using a computer.
\begin{lemma} \label{lemma:manifold_Y}
A computable upper bound for $\lVert A F (\bar{w}) \rVert_{X_{\sTF}}$ is given by
    \begin{equation}
        \label{eq:manifold_Y}
        \begin{aligned}
            Y(\bar{w}, r_{\sF}) \bydef &
            r_{\sF} \lVert A_{f} \rVert_{B(X_{\sTF})}
            +
            \lVert  A_{f} \proj{[2,N]}{T} \proj{[0,M]}{F} F(\bar{w}) \rVert_{X_{\sTF}}
            \\&
            + r_{\sF}
            + \lVert\proj{(N,3N]}{T} \proj{[0,3M]}{F}  L_{\lambda}^{-1} F (\bar{w})\rVert_{X_{\sTF}}
            + \lVert\proj{[2,N]}{T} \proj{(M,3M]}{F}  L_{\lambda}^{-1} F (\bar{w})\rVert_{X_{\sTF}}
            .
        \end{aligned}
    \end{equation}
\end{lemma}
\begin{proof}
    Since the vector field $f(\bar{w})$ includes a cubic term, an analogous result to Lemma~\ref{lemma:convolution_product} implies that
    \begin{equation*}
        \begin{aligned}
            F (\bar{w}) & =    \proj{[0,1]}{T} F (\bar{w})
            +
            \proj{[2,\infty)}{T} F(\bar{w})
            =
            \proj{[0,1]}{T} B (\bar{w},v) +
            \proj{[2,N]}{T} \proj{[0,M]}{F}  L_{\lambda} \bar{w} - \proj{[2,3N]}{T} \proj{[0,3M]}{F} f(\bar{w})
            \\ &
            =
            \proj{[0,1]}{T} B (\bar{w},v)
            +
            \proj{[2,3N]}{T} \proj{[0,3M]}{F} F(\bar{w}).
        \end{aligned}
    \end{equation*}
    Therefore 
    \begin{equation*}
        \begin{aligned}
            A_{\infty} F (\bar{w}) & =
            \proj{[0,1]}{T}  \proj{(M,\infty)}{ F} B (\bar{w},v)
            +
            \left( \proj{(N,\infty)}{T} L_{\lambda}^{-1} \proj{[2,3N]}{T} \proj{[1,3M]}{F}  +
            \proj{[2,N]}{T}  \proj{(M,\infty)}{ F}L_{\lambda}^{-1} \proj{[2,3N]}{T} \proj{[1,3M]}{F}
            \right) F(\bar{w})
            \\& =
            \proj{[0,1]}{T}  \proj{(M,\infty)}{ F} B (\bar{w},v)
            +
            \left(
            \proj{(N,3N]}{T} \proj{[0,3M]}{F} +
            \proj{[2,N]}{T} \proj{(M,3M]}{F}
            \right) L_{\lambda}^{-1} F(\bar{w})
        \end{aligned}
    \end{equation*}
    while the evaluation of the operator $A_f$ is given by
    \[
        A_{f}F (\bar{w})  =  A_{f}\proj{[0,1]}{T} \proj{[0,M]}{F} B (\bar{w},v)
        +
        A_{f}\proj{[2,N]}{T} \proj{[0,M]}{F} F(\bar{w}).
    \]
    Recall that we do not know the exact value of \(\lambda\), but we know it lies in an interval of radius \(r_{\sF}\) centered at \(\bar{\lambda}\). This creates no computability issues because, as discussed earlier, we use interval arithmetic to perform computations on intervals. In this setting, computing the norm of a computable expression involving \(\lambda\) produces an interval, whose upper boundary provides a rigorous upper bound for the computed norm. On the other hand, the term \( B(\bar{w}, v)\) depends on the solution $v$ of the bundle problem \eqref{eq:bundle_coeff}, an infinite sequence. Since \(v\) lies within the ball in  $S_{\sF}^2$  of radius \(r_{\sF}\) centered at \(\bar{v}\), there exists a \(\delta\)  such that
    \(
    v = \bar{v} + \delta, \)  with
    \( \lVert \delta \rVert_{S_{\sF}^2} \leq r_{\sF}.
    \)
    From the definition of $B$ in \eqref{eq:manifold_F_B} and the definition of $\bar{w}$ in \eqref{eq:w_bar_0_1} follows that
    \begin{equation}
        [B^{(i)}(\bar{w},v)]_{n,m} =
        \begin{cases}
            \bar{v}_{m}^{(i)} - v^{(i)}_m, & n=1, m \in \mathbb{Z}      \\
            0,                             & n \neq 1, m \in \mathbb{Z}
        \end{cases}
        \quad
        \text{for}
        \quad i=1,2.
    \end{equation}
    Therefore,
    \begin{align*}
        \lVert B(\bar{w},v) \rVert_{X_{\sTF}}
         & =
        \max  \left \{  \lVert B^{(1)}(\bar{w},v) \rVert_{{\sTF}}  , \lVert B^{(2)}(\bar{w},v) \rVert_{{\sTF}}    \right\}
        \\ & =
        \max  \left \{   \lVert \bar{v}^{(1)} - v^{(1)} \rVert_{{\sF}}  ,  \lVert \bar{v}^{(2)} - v^{(2)}\rVert_{{\sF}}    \right\}
        \\ & =
        \max  \left \{   \lVert \delta^{(1)} \rVert_{{\sF}}  ,  \lVert \delta^{(2)}\rVert_{{\sF}}    \right\}
        \leq r_{\sF}.
    \end{align*}
    Combining the above observations, we get that a computable bound for \(\lVert AF (\bar{w}) \rVert_{X_{\sTF}}\) is given by \eqref{eq:manifold_Y}.
\end{proof}

Next, we compute a $Z_1$ bound satisfying \eqref{eq:Z2_bound_general}.

\begin{lemma}
    \label{lemma:manifold_Z1}
    A computable upper bound for $\lVert  I - A D F (\bar{w})  \rVert_{B(X_{\sTF})}$ is given by
    \begin{align} \label{eq:manifold_Z1}
        Z_1(\bar{w}, \bar{x}_{\sF}) \bydef & \left\|  \proj{[0,N]}{T}\proj{[0,M]}{F} -  A_{f} D F(\bar{w})\proj{[0,3N]}{T} \proj{[0,3M]}{F} \right\|_{B(X_{\sTF})} \\
                                           & + \left(
        \frac{1}{\sqrt{4{\lambda}^2 + (M+1)^2}}
        + \frac{1}{|{\lambda}(N+1)|} \right)  \max \left\{ 1 ,
        |a|  +
        |b|\nu^2
        +  3|c|\lVert\bar{w}^{(1)}*_{\sTF}\bar{w}^{(1)} \rVert_{\sTF} \right\}.
        \nonumber
    \end{align}
\end{lemma}
\begin{proof}
    The explicitly expression for the derivative of the validation map is given by
    \[
        D F (\bar{w})
        =
        \proj{[0,1]}{T}
        +
        \proj{[2,\infty)}{T} \left[ L_{\lambda} -  Df (\bar{w}) \right],
    \]
    where the action of the derivative on a vector $h=(h^{(1)},h^{(2)})$ is given by
    \[
        Df(\bar{w})h =
        \begin{pmatrix}
            h^{(2)}                                                          \\
            -  ah^{(1)} +
            b\left(w^{(3)}*_{\sTF}h^{(1)}\right) +
            3c\left(\bar{w}^{(1)}*_{\sTF}\bar{w}^{(1)}*_{\sTF}h^{(1)}\right) \\
        \end{pmatrix}.
    \]
    Using that the sequence product \eqref{eq:Taylor_Fourier_product} satisfies a Banach algebra property such as \eqref{eq:convolution_property},
    \begin{equation} \label{eq:bound_derivative_manifold}
        \lVert Df(\bar{w})\rVert_{B(X_{\sTF})} = \sup_{\lVert h\rVert_{{\sTF}}\leq1} \lVert Df(\bar{w})h\rVert_{X_{\sTF}}
        \le
        \max \left\{ 1 ,
        |a|  +
        |b|\lVert w^{(3)}  \rVert_{\sTF}
        +  3|c|\lVert\bar{w}^{(1)}*_{\sTF}\bar{w}^{(1)} \rVert_{\sTF} \right\}.
    \end{equation}
    Observe that $\lVert w^{(3)}  \rVert_{\sTF}= \lVert \gamma^{(3)}  \rVert_{\sF}= \nu^2$.
    We use an analogous splitting for $ I - A D F(\bar{w})$ as in Section \ref{sec:bundle} in equation \eqref{eq:Z1_splitting}.
    \begin{equation}
        \label{eq:manifold_splitting}
        \begin{aligned}
            \nonumber
            I - A D F(\bar{w}) & =
            \proj{[0,N]}{T}\proj{[0,M]}{F} -  A_{f} D F(\bar{w})\proj{[0,3N]}{T} \proj{[0,3M]}{F}
            - \
            A_{f} D F(\bar{w}) \left[ \proj{[0,3N]}{T}\proj{(3M,\infty)}{F} + \proj{(3N,\infty)}{T} \right]
            \\ & \quad +
            \proj{[0,N]}{T}\proj{(M,3M]}{F}
            + \proj{(N,3N]}{T} \proj{[0,3M]}{F} -
            A_{\infty} D F(\bar{w})\proj{[0,3N]}{T} \proj{[0,3M]}{F}
            \\& \quad +
            \proj{[0,N]}{T}\proj{(3M,\infty)}{F}
            +
            \proj{(N,3N]}{T}\proj{(3M,\infty)}{F}
            + \proj{(3N,\infty)}{T} -
            A_{\infty} D F(\bar{w}) \left[ \proj{[0,3N]}{T}\proj{(3M,\infty)}{F} + \proj{(3N,\infty)}{T} \right].
        \end{aligned}
    \end{equation}
    Analogous properties to those in Lemma~\ref{lemma:convolution_product} hold for the product in the space $S_{\sTF}$. In this case, there is a cubic product term in $f(\bar{w})$, so for the finite part of $A$ evaluated at the infinite tail of the sequence we have
    \begin{align}
        \nonumber
        A_{f} D F(\bar{w}) \left[ \proj{[0,3N]}{T}\proj{(3M,\infty)}{F} + \proj{(3N,\infty)}{T} \right]
         & =
        A_{f} \proj{[0,N]}{T}\proj{[0,M]}{F} Df(\bar{w}) \left[  \proj{[0,3N]}{T}\proj{(3M,\infty)}{F} +  \proj{(3N,\infty)}{T}   \right] = 0.
    \end{align}
    Now, we look at the infinite part of the operator evaluated at the finite part of the sequences. It is straightforward to show that
    \begin{equation*}
        \begin{aligned}
            \proj{[0,N]}{T}\proj{(M,3M]}{F}
            + \proj{(N,3N]}{T} \proj{[0,3M]}{F} -
            A_{\infty} D F(\bar{w})\proj{[0,3N]}{T} \proj{[0,3M]}{F}
             & =
            \proj{[2,N]}{T} \proj{(M, \infty)}{F} L_{\lambda}^{-1} Df(\bar{w}) \proj{[0,3N]}{T} \proj{[0,3M]}{F}
            +
            \proj{(N,\infty)}{T} L_{\lambda}^{-1} Df(\bar{w}) \proj{[0,3N]}{T} \proj{[0,3M]}{F}.
        \end{aligned}
    \end{equation*}
    Similarly, for the infinite part of the operator evaluated at the infinite tail of the sequence we have
    \begin{align*}
         & \hspace{-1cm} \proj{[0,N]}{T}\proj{(3M,\infty)}{F} + \proj{(N,3N]}{T}\proj{(3M,\infty)}{F}
        + \proj{(3N,\infty)}{T} -
        A_{\infty} D F(\bar{w}) \left[ \proj{[0,3N]}{T}\proj{(3M,\infty)}{F} + \proj{(3N,\infty)}{T} \right]
        \\
         & =
        \proj{[2,N]}{T} \proj{(M, \infty)}{F} L_{\lambda}^{-1} Df(\bar{w}) \left[  \proj{[0,3N]}{T} \proj{(3M,\infty)}{F}  + \proj{(3N,\infty)}{T}  \right]
        + \proj{(N,\infty)}{T} L_{\lambda}^{-1} Df(\bar{w}) \left[   \proj{[0,3N]}{T} \proj{(3M,\infty)}{F} + \proj{(3N, \infty)}{T} \right].
    \end{align*}
    Putting all together we have
    \begin{equation*}
        \begin{aligned}
            \lVert I - D F(\bar{w}) \rVert_{B(X_{\sTF})}
             & \leq
            \lVert  \proj{[0,N]}{T}\proj{[0,M]}{F} -  A_{f} D F(\bar{w})\proj{[0,3N]}{T} \proj{[0,3M]}{F} \rVert_{B(X_{\sTF})}
            \\  & +
            \lVert \proj{[2,N]}{T} \proj{(M, \infty)}{F} L_{\lambda}^{-1} Df(\bar{w})
            \rVert_{B(X_{\sTF})}
            +
            \lVert \proj{(N,\infty)}{T} L_{\lambda}^{-1} Df(\bar{w})
            \rVert_{B(X_{\sTF})}
            \\  & \leq
            \lVert  \proj{[0,N]}{T}\proj{[0,M]}{F} -  A_{f} D F(\bar{w})\proj{[0,3N]}{T} \proj{[0,3M]}{F} \rVert_{B(X_{\sTF})}
            \\  & +
            \left(
            \lVert \proj{[2,N]}{T} \proj{(M, \infty)}{F} L_{\lambda}^{-1} \rVert _{B(X_{\sTF})}
            +
            \lVert \proj{(N,\infty)}{T} L_{\lambda}^{-1} \rVert _{B(X_{\sTF})}
            \right)
            \lVert Df(\bar{w})\rVert_{B(X_{\sTF})} .
        \end{aligned}
    \end{equation*}
    %
    The first term is already a computable expression. To deal with the second term we use the following bounds:
    \begin{equation}
        \label{eq:manifold_tail_bound}
        \lVert  \proj{[2,N]}{T}  \proj{(M,\infty)}{F} L_{\lambda}^{-1} \rVert_{B(X_{\sTF})} \leq
        \frac{1}{\sqrt{4{\lambda}^2 + (M+1)^2}},
        \quad
        \lVert \proj{(N,\infty)}{T} L_{{\lambda}}^{-1} \rVert_{B(X_{\sTF})} \leq \frac{1}{|{\lambda}(N+1)|}.
    \end{equation}
    Therefore, a computable upper bound for $\lVert  I - A D F (\bar{w})  \rVert_{B(X_{\sTF})}$ is given by \eqref{eq:manifold_Z1}.
\end{proof}

\begin{lemma} \label{lemma:manifold_Z2}
    A computable $Z_2$ such that   $\lVert A(DF(w) - DF(\bar{w})) \rVert_{B(X_{\sTF})} \leq Z_2 r$  for all $x$ in $B(\bar{w}_{\sTF}, r)$ is given by
    \begin{equation} \label{eq:manifold_Z2}
        Z_2(\bar{w}, r) \bydef
        3 |c| \left(
        \lVert A_{f} \rVert_{B(X_{\sTF})}
        +
        \frac{1}{\sqrt{4{\lambda}^2 + (M+1)^2}}
        +
        \frac{1}{|{\lambda}(N+1)|}
        \right)
        \left( 2\lVert \bar{w}_{\sTF}^{(1)}\rVert_{\sTF}  + r \right).
    \end{equation}
\end{lemma}
\begin{proof}
    Let $w$ be and element in  $B(\bar{w}_{\sTF}, r)$. Hence, there exists a $ \mathbf{\delta} \in X_{\sTF}  $ such that $ w =\bar{w} + \delta$ with  $\lVert \mathbf{\delta}  \lVert_{X_{\sTF}}  \leq r $. Let $h$ in $X_{\sTF}$ be such that $\lVert h  \lVert_{X_{\sTF}}  \leq 1 $ and let $z \bydef [ DF(\bar{w} + \mathbf{\delta}) - DF(\bar{w})] h $. Since
    \[
        [Df(\bar{w}) - Df(\bar{w} + \mathbf{\delta})]h =
        \begin{pmatrix}
            0                                                                                                                    \\
            -3c\left(2\bar{w}^{(1)}*_{\sTF}\delta^{(1)}*_{\sTF}h^{(1)} + \delta^{(1)}*_{\sTF}\delta^{(1)}*_{\sTF}h^{(1)} \right) \\
        \end{pmatrix},
    \]
    we have
    \[
        \lVert  z \rVert_{X_{\sTF}} =
        \lVert  \proj{[2,\infty)}{T} \left[ Df(\bar{w}) - Df(\bar{w} + \mathbf{\delta}) \right]h \rVert_{X_{\sTF}} \leq r \left( 6|c|\lVert \bar{w}_{\sTF}^{(1)}\rVert_{\sTF}  + 3|c|r \right).
    \]
    And thus
    \begin{align}
        \nonumber
        \lVert Az  \rVert_{X_{\sTF}}
         & \leq
        \lVert A_{f} z \rVert_{X_{\sTF}}
        +
        \lVert \proj{[2,N]}{T}  \proj{(M,\infty)}{ F}L_{\lambda}^{-1} z\rVert_{X_{\sTF}}
        +
        \lVert \proj{(N,\infty)}{T}  L_{\lambda}^{-1}  z\rVert_{X_{\sTF}}
        .
    \end{align}
    Observe that in this case the bound for $Z_2$ depends on the radius $r$ of the ball.
    Therefore, using the bounds \eqref{eq:manifold_tail_bound} a computable bound  $Z_2(\bar{w}, r)$ is given by \eqref{eq:manifold_Z2}.
\end{proof}
Using the computable bounds established above, we can formulate a validation theorem for an approximate parameterization of the stable manifold associated with the periodic orbit \( \gamma \).
%

\begin{theorem}
    \label{Theorem:Manifold}
    Fix parameters \(a\), \(b\), and \(c\), a weight \(\nu\) for the norm in \eqref{eq:S_F}. We consider a solution of the bundle sequence equation \eqref{eq:bundle_coeff} of the form $x_{\sF} \bydef (\lambda, v^{(1)}, v^{(2)}, 0, 0)$ in the space $\mathbb{C} \times S^{4}_{\sF} $ where $ (\lambda, v^{(1)}, v^{(2)}  )$ is in the ball $B_{X_{\sF}}(\bar{x}_{\sF}, r_{\sF})$. Suppose we have   $\bar{w}$ in $X_{\sTF}$ such that
    \begin{equation}
        \label{eq:initial_guess_manifold}
        w  = \proj{[0,N]}{T} \proj{[0,M]}{F} \bar{w}
    \end{equation}
    for fixed Taylor and Fourier truncation modes given by $N$ and $M$, respectively. Additionally, assume that the matrix \(\bar{A}(\bar{x}_{\sF})\) from operator \eqref{eq:manifold_A} is computed as a numerical inverse of
    \[
        \proj{[0,N]}{T}   \proj{[0,M]}{F}
        DF(\bar{w})\proj{[0,N]}{T}   \proj{[0,M]}{F}.
    \]
    Suppose the computable bounds \(Y(\bar{w}, r_{\sF})\), \(Z_1(\bar{w}, \bar{x}_{\sF})\), and \(Z_2(\bar{w}, r_{\sTF}^{*})\) defined in \eqref{eq:manifold_Y}, \eqref{eq:manifold_Z1}, and \eqref{eq:manifold_Z2} satisfy
    \[
        Z_{1}(\bar{w}, \bar{x}_{\sF}) < 1 \quad \text{and} \quad Z_2(\bar{w}, r_{\sTF}^{*})  < \frac{(1 - Z_{1}(\bar{w}, \bar{x}_{\sF}))^{2}}{2Y(\bar{w}, r_{\sF})}
    \]
    for some $r_{\sTF}^{*}$. Then, the validation map \eqref{eq:manifold_F} has a unique zero $w=(w^{(1)}, w^{(2)} )$ in the ball $B_{X_{\sTF}}( \bar{w} , r_{\sTF})$ with radius
    \[
        r_{\sTF} \bydef \frac{1 - Z_{1}(\bar{w}, \bar{x}_{\sF})- \sqrt{ (1- Z_1(\bar{w}, r_{\sTF}^{*}))^2 - 2Y(\bar{w}, r_{\sF})
                Z_2(\bar{w}, r_{\sTF}^{*})}}{Z_2(\bar{w}, r_{\sTF}^{*})}
    \]
    such that $F(w)= 0$.
    Furthermore, assume that the coefficients of \(\bar{w}\) satisfy the symmetry condition
    \[
        \bar{w}_{n,k}^{(i)} = [\bar{w}_{n,-k}^{(i)}]^{*},  \quad k \in \Z.
    \]
    Then, there exists a parameterization of the stable manifold $W: [-1,1] \times [0,2\pi] \to \mathbb{R}^{4}$ attached to the periodic orbit $\gamma$ such that
    \begin{equation}
        \label{eq:parameterization_bound}
        |W^{(i)}(\theta, \sigma) - \bar{W}^{(i)}(\theta, \sigma)    \bydef
        \sum_{n = 0}^{N}\sum_{m = -M }^{M}
        \bar{w}_{n,m }^{(i)}
        e^{im\theta} \sigma^{n}| \leq  r_{\sTF}
        \quad
        \text{for}
        \quad
        i=1,2.
    \end{equation}
\end{theorem}

\begin{proof}
    The bounds $Y(\bar{w}, r_{\sF})$, $Z_{1}(\bar{w}, \bar{x}_{\sF})$ and $Z_2(\bar{w},r_{\sTF}^{*})$ given by equations \eqref{eq:manifold_Y}, \eqref{eq:manifold_Z1} and \eqref{eq:manifold_Z2} satisfy the hypothesis of Theorem \ref{Theorem:Newton} for some $r_{\sTF}^{*}$. There exists a unique zero $w=(w^{(1)}, w^{(2)})$  of the validation map \eqref{eq:manifold_F} in the ball $B_{X_{\sTF}}(\bar{w}, r_{\sTF})$ for some radius $r_{\sTF} \in \R$.
    Then a solution to the sequence equation \eqref{eq:manifold_coeff} associated to the PDE problem \eqref{eq:manifold_PDE} with parameters $a$, $b$ and $c$ is $x_{\sTF} \bydef (w^{(1)}, w^{(2)}, w^{(3)}, w^{(4)}  ) \in   S^{4}_{\sTF} $. Moreover, since the Fourier expansions of the approximate solution $\bar{w}$ satisfy the symmetry $\bar{w}_{n,k}^{(i)} = [\bar{w}_{n,-k}^{(i)}]^{*}$, $k \in \Z$, analogously to the proof of Theorem~\ref{Theorem:Bundle} we can conclude that the functions $W_{n}^{(i)}$ are real valued  for $i=1,2$, $n\in\N$. The resulting parameterization for the stable manifold $W$ attached to the periodic orbit $\gamma$ is given by
    \[
        W(\theta, \sigma) \bydef
        \sum_{n = 0}^{\infty}
        \sum_{m \in \mathbb{Z}}
        w_{n,m } e^{im\theta} \sigma^{n}
    \]
    where the third and fourth components are given by \eqref{eq:manifold_3_4}.
    Finally, since $\nu \ge 1$ and $|\sigma|\leq1$ for $i = 1,2 $ we have that
    \begin{equation}
        \begin{aligned}
            |W^{(i)}(\theta, \sigma) - \bar{W}^{(i)}(\theta, \sigma)| & \leq
            \lVert \bar{w}^{(i)} - w^{(i)} \rVert_{{\sTF}}
            \leq
            r_{\sTF}
            . \qedhere
        \end{aligned}
    \end{equation}
\end{proof}
We can only guarantee that the output of a parameterization coming from Theorem \ref{Theorem:Manifold} is included in a ball of some radius. However, this fact does not represent a computational limitation. We are now ready to solve the boundary-value problem \eqref{eq:BVP_soliton}.

\section{Solving the Boundary-Value Problem} \label{sec:BVP}


Having detailed in the previous section the computation of a parameterization \( W \) of the stable manifold associated with the periodic orbit $\gamma$, we can now reformulate the boundary-value problem \eqref{eq:BVP_soliton} as follows

\begin{equation}
    \label{eq:BVP}
    \dot{u}
    =
    \kappa
    f(u),
    \quad
    \kappa \bydef \frac{\Lr}{2},
\end{equation}
\begin{equation}
    \nonumber
    u_{2}(-1) = 0,
    \quad
    u_{3}(-1) = 1,
    \quad
    u_{4}(-1) = 0,
    \quad
    u_{1}(1) = W_{1}(\theta, \sigma),
    \quad
    u_{2}(1) = W_{2}(\theta, \sigma).
\end{equation}
We have 5 boundary conditions for the 4 components of the solution. To balance the system we will solve for $\sigma$ and fix the value of $\theta$ and $\Lr$.
We now consider the following Chebyshev expansion for the solution of the BVP \eqref{eq:BVP}
\begin{equation}
\label{eq:sol_expansion}
    s^{(i)}(t) = s^{(i)}_0 +  2 \sum_{m\geq1} s^{(i)}_m T_m(t), \quad i = 1,2,3,4,
\end{equation}
where $T_m:[-1,1] \to \R $ with $m \ge 0$ are the Chebyshev polynomials of the first kind. These polynomials are defined recursively as follows:
\[
T_0(t) = 1, \quad T_1(t) = t, \quad T_{m+1}(t) = 2t T_m(t) - T_{m-1}(t) \quad \text{for } m \geq 1.
\]
Given a weight $\omega \ge 1$ we define the sequence space $S_{\sC}$ of Chebyshev coefficients as
\begin{equation} \
    S_{\sC}\bydef \left\{ s = (s_m)_{m \ge 0}, s_m \in \R : \|s\|_{\sC} \bydef |s_0| + 2 \sum_{m \ge 1} |s_m| \omega^m < \infty \right\},
    \label{eq:S_C}
\end{equation}
and define the Banach space $X_{\sC}$ as
\begin{equation}
    X_{\sC} \bydef \R \times S_{{\sC}}^4,
    \quad
    \| (\sigma,s) \|_{X_{\sC}} \bydef \max \left\{ |\sigma|,\|s^{(1)}\|_{\sC},\|s^{(2)}\|_{\sC},\|s^{(3)}\|_{\sC},\|s^{(4)}\|_{\sC} \right\}.
\end{equation}
To easily make reference to the different components of the space $X_{\sC}$ we use truncation operators. For any $p \in S_{\sC}$ and a set of indices $R \subset \N$ we define the projection operator as
\[
    \left[ \proj{R}{C}  p \right]_{m}\bydef \begin{cases}
        p_m, & m\in R            \\
        0,   & \text{otherwise},
    \end{cases}
    \quad
    \proj{R}{C} (\sigma,p) \bydef (0,  \proj{R}{C} p),
    \quad \text{and} \quad
    \proj{}{\mathbb{R}} (\sigma,p) \bydef (\sigma,0).
\]

The action of a truncation operators in a element $(\sigma,s)$ in the space $X_{\sC} $ is given by
\[
    \proj{R}{C} (\sigma,s) \bydef (0,  \proj{R}{C} s),
    \quad
    \proj{}{\mathbb{R}} (\sigma,s) \bydef (\sigma,0)
    \quad \text{and} \quad
    \proj{R}{C} s \bydef
    \left(
    \proj{R}{C} s^{(1)},
    \proj{R}{C} s^{(2)},
    \proj{R}{C} s^{(3)},
    \proj{R}{C} s^{(4)}
    \right).
\]
%
For any two sequences $u,v \in S_{\sC}$, we define their discrete convolution $*_{\sC} :S_{\sC} \times S_{\sC} \to S_{\sC}$ by
\begin{equation} \label{eq:discrete_convolution_cheb}
    (u *_{\sC} v)_m \bydef \sum_{m_1+m_2=m \atop m_1,m_2 \in \Z} u_{|m_1|}  v_{|m_2|}.
\end{equation}

For any sequence $s \in S_{\sC}$ we will use the following notation when we refer to the sequence as a function
\[
    s(t) \bydef  s_0 +  2 \sum_{m\geq1} s_m T_m(t).
\]
Since Chebyshev polynomials satisfy the identity \( T_k(\cos \theta) = \cos(k\theta) \), we have that $T_{k}(-1)=(-1)^{k}$ and $T_{k}(1)=1$ for all $k$. Therefore, the evaluations at $-1$ and $1$  of a sequence are given by
\[
    s(-1) =  s_{0} + 2\sum_{m \geq 1 } (-1)^m s_{m}, \quad s(1) =  s_{0} + 2\sum_{m \geq 1 } s_{m}.
\]

We proceed with the definition of the validation map for the soliton boundary-value problem. We employ the same notation as in previous sections to highlight the structural parallels in our approach. Moreover, as derived in Section~\ref{sec:manifold}, let $W: [-1,1] \times [0,2\pi] \to \mathbb{R}^{4}$ be a parameterization of the local stable manifold $W^s_{\rm loc}(\gamma)$ associated with the periodic orbit $\gamma$. 

In this case, the sequence equation follows the derivation in \cite{Lessard2014}. For completeness, we briefly outline the argument for the first component of \( s \), which satisfies \( s^{(1)}(1) = W^{(1)}(\theta, \sigma) \). We begin by integrating the differential equation \( \dot{s}(t) = f(s(t)) \) from \( t \) to \( 1 \):
\begin{equation}
\label{eq:integral_identity}
s^{(1)}(1) - s^{(1)}(t) = \kappa \int_{t}^{1} f^{(1)}(s(u)) \, du.
\end{equation}
We then substitute in \eqref{eq:integral_identity} the Chebyshev expansion of the solution, given in equation \eqref{eq:sol_expansion}, along with the Chebyshev expansion of the vector field \( f \):
\[
f^{(i)}(s(t)) = \phi^{(i)}_0 + 2 \sum_{m \geq 1} \phi^{(i)}_m T_m(t), \quad i = 1, 2, 3, 4.
\]
To evaluate the integral in \eqref{eq:integral_identity}, we use the recurrence relation for integrals of Chebyshev polynomials:
\[
\int T_m(t)\,dt = \frac{1}{2} \left( \frac{T_{m+1}(t)}{m+1} - \frac{T_{m-1}(t)}{m-1} \right), \quad \text{for } m \geq 2.
\]
Matching Chebyshev coefficients on both sides yields a system of equations for the unknown coefficients \( a_m \). This leads to the following equations for the Chebyshev coefficients:
\[
s_0 + 2 \sum_{m \geq 1} (-1)^m s_m - W^{(1)}(\theta, \sigma)=0,
\]
\[
2m\, s_m^{(1)} + \phi^{(1)}_{m+1} - \phi^{(1)}_{m-1} = 0, \quad \text{for } m \geq 1.
\]
Since the Chebyshev expansion of a sum or difference of two functions is obtained by adding or subtracting their Chebyshev coefficients, and the Chebyshev coefficients of a product are given by the discrete Chebyshev convolution~\eqref{eq:discrete_convolution_cheb}, any polynomial in the components of \( s \) can be written using sums and discrete convolutions of their Chebyshev coefficients. In particular, the Chebyshev coefficients of \( f(s) \) can be expressed directly in terms of those of \( s \).

Applying the same procedure to the remaining three components of \( s \), we obtain a complete sequence equation.  Further details on the derivation of the sequence equations can be found in \cite{Lessard2014}. This construction defines our validation map.
\begin{equation}
    \label{eq:soliton_F}
    F(x) \bydef \begin{pmatrix} 0 \\ L s + \kappa \ T \fop (s)
    \end{pmatrix}
    + B(x)
\end{equation}
where the operators $L$, $T$, $f$ and $B$ are defined below.
The operator $f : S_{\sC}^4 \to S_{\sC}^4$ is defined as
\[
    f (s) \bydef
    \begin{pmatrix}
        s^{(2)}                                                                                               \\
        -as^{(1)} +  b\left(s^{(3)}*_{\sC}s^{(1)}\right) +  c\left(s^{(1)}*_{\sC}s^{(1)}*_{\sC}s^{(1)}\right) \\
        s^{(4)}                                                                                               \\
        -4s^{(3)}
    \end{pmatrix}
\]
while $L:S_{\sC}^4 \to S_{\sC}^4$ is the linear operator defined as follows
\[
    \left[L s\right]_{m}^{(i)} \bydef
    \begin{cases}
        0,            & m = 0             \\
        2m s_m^{(i)}, & \text{otherwise},
    \end{cases}
    \quad
    \left[L^{-1} s\right]_{m}^{(i)}  \bydef
    \begin{cases}
        0,                      & m = 0            \\
        \frac{1}{2m} s_m^{(i)}, & \text{otherwise}
    \end{cases}
    \quad
    i = 1,2,3,4.
\]
We also define the tridiagonal sequence operator $T: S_{\sC}^4 \to S_{\sC}^4$ that comes from the sequence expansion in Chebyshev series
\[
    \left[T  s\right]_{m}^{(i)} \bydef
    \begin{cases}
        0,                & m = 0             \\
        -s_{m-1}+s_{m+1}, & \text{otherwise}.
    \end{cases}
\]
Finally, we define the sequence operator $B :X_{\sC} \to X_{\sC}$ that includes the boundary conditions in the zero-finding problem. Each sequence has only the zero coefficient different than zero. While the  parameter space component includes the remaining boundary condition
\begin{align*}
    \proj{[0,\infty)}{C} B(x) = \proj{\{0\}}{C} B(x) & \bydef
    \left(
    s^{(1)}(1) - W^{(1)}(\theta,\sigma) ,
    s^{(2)}(1) - W^{(2)}(\theta,\sigma),
    s^{(3)}(-1)  - 1    ,
    s^{(4)}(-1)
    \right),
    \\
    \proj{}{R} B(x)                                  & \bydef (s^{(2)}(-1),0).
\end{align*}
We now consider $\bar{x}_{\sC} = (\bar{\sigma}, \bar{s})$  an element in $X_{\sC}$ supported in $[0,M)$ and such that $-1<\bar{\sigma}<1$.
We define an approximate inverse derivative $A: X_{\sC}  \to X_{\sC} $ as follows
\[
    A \bydef A_{f} + A_{\infty},
    \quad
    A_{f} \bydef
    \left( \proj{}{\mathbb{C}} + \proj{[0, M]}{C} \right)
    \bar{A}(\bar{x}_{\sC}) \left( \proj{}{\mathbb{C}} + \proj{[0, M]}{C} \right)
    ,
    \quad
    A_{\infty} \bydef
    \begin{pmatrix} 0 & 0 \\ 0 &
                L^{-1} \proj{(M,\infty)}{C}
    \end{pmatrix},
\]
where \( A_f \) can be represented as a \( [4(M+1)+1] \times [4(M + 1)+1] \) matrix that in practice corresponds to a numerical inverse of the truncated derivative of $F$ evaluated at $\bar{x}_{\sC}$. We suppose that for a fixed set of parameters $a$, $b$ and $c$ there exists $\bar{x}_{\sTF}  = (\bar{w}^{(1)}, \bar{w}^{(2)}, \bar{w}^{(3)}, \bar{w}^{(4)}  ) \in   S^{4}_{\sTF} $  with finite support such that there exists a parameterization of the stable manifold $W: [-1,1] \times [0,2\pi] \to \mathbb{R}^{4}$ attached to the periodic orbit $\gamma$ satisfying the inequality in equation \eqref{eq:parameterization_bound}.
This is the parameterization given by Theorem~\ref{Theorem:Manifold}. We are now ready to provide computable bounds as those required by Theorem \ref{Theorem:Newton}.
\begin{lemma} \label{lemma:soliton_Y}
    A computable upper bound for $\lVert A F (\bar{x}_{\sC}) \rVert_{X_{\sC}}$ is given by
    \begin{align} \label{eq:soliton_Y}
        Y(\bar{x}_{\sC}, r_{\sTF}) & \bydef \left\|
        A_{f} \begin{pmatrix} \bar{s}^{(2)}(-1) \\ L \bar{s} + \kappa \ T \fop (\bar{s}) \end{pmatrix} \right\|_{X_{\sC}} + \lVert L^{-1} \proj{(M,3M+1]}{C} F(\bar{x}_{\sC})\rVert_{X_{\sC}}
        \\ & +
        \lVert A_{f} \rVert_{B(X_{\sC})} \max \left\{ \right.  |\bar{s}^{(1)}(1)-\bar{W}^{(1)}(\theta, \bar{\sigma}) | +  r_{\sTF} ,
        |\bar{s}^{(2)}(1)-\bar{W}^{(2)}(\theta, \bar{\sigma}) | +  r_{\sTF} ,
        |\bar{s}^{(3)}(-1)-1|  , |\bar{s}^{(4)}(-1)| \left. \right\}.
        \nonumber
    \end{align}
\end{lemma}
\begin{proof}
    Since $f(\bar{s})$ includes a cubic product and operator $T$ augments the support by one, we know that $F(\bar{x}_{\sC})$ is supported in $[0,3M+1]$. Thus, using the triangle inequality we have that
    \begin{equation}
        \lVert  A F(\bar{x}_{\sC})\rVert_{X_{\sC}}
        \leq
        \lVert A_{f} F(\bar{x}_{\sC})\rVert_{X_{\sC}}
        +  \lVert L^{-1} \proj{(M,3M+1]}{C} F(\bar{x}_{\sC})\rVert_{X_{\sC}}.
    \end{equation}
    The second term is already a computable bound. The first term, however, involves the evaluation of \(B(\bar{x}_{\sC})\), which requires evaluating the parameterization \(W\). Since we only have interval control over \(W\), it is not directly a computable bound. Indeed,
    \[
        A_{f} F(\bar{x}_{\sC}) =   A_{f}
        \begin{pmatrix} \bar{s}^{(2)}(-1) \\ L \bar{s} + \kappa \ T \fop (\bar{s})
        \end{pmatrix}
        +
        A_{f}     \begin{pmatrix} 0 \\
            \proj{[0,\infty)}{C}B(\bar{x}_{\sC})
        \end{pmatrix} .
    \]
    Using \eqref{eq:parameterization_bound}, it follows that
    \begin{align*}
        \left\| \proj{[0,\infty)}{C}  B(\bar{x}_{\sC}) \right\|_{X_{\sTF}}
         & =
        \max_{i=1,2,3,4}  \left \{ \left\| [ \proj{\{0\}}{C} B(\bar{x}_{\sC})] ^{(i)}\right\|_{\sC}     \right\}
        \\ & =
        \max \left\{ \right.  |\bar{s}^{(1)}(1) - W^{(1)}(\theta, \bar{\sigma}) | ,
        |\bar{s}^{(2)}(1) - W^{(2)}(\theta, \bar{\sigma}) |  ,
        |\bar{s}^{(3)}(-1)- 1|,
        |\bar{s}^{(4)}(-1)| \left. \right\}
        \\ & \leq
        \max \left\{ \right.  |\bar{s}^{(1)}(1) - \bar{W}^{(1)}(\theta, \bar{\sigma}) | +  r_{\sTF} ,
        |\bar{s}^{(2)}(1) - \bar{W}^{(2)}(\theta, \bar{\sigma}) | +  r_{\sTF} ,
        |\bar{s}^{(3)}(-1)  - 1|,
        |\bar{s}^{(4)}(-1)| \left. \right\}.
    \end{align*}
    Therefore, we define the bound $Y$ as in equation \eqref{eq:soliton_Y}.
\end{proof}

We now provide a computable $Z_1$ bound
\begin{lemma} \label{lemma:soliton_Z1}
    A computable upper bound for $\lVert  I - A D F (\bar{x}_{\sC})\rVert_{B(X_{\sC})}$ is given by
    \begin{align} \label{eq:soliton_Z1}
        Z_{1}(\bar{x}_{\sC}, r_{\sTF}) & \bydef
        \lVert
        \proj{}{\mathbb{R}} \ + \proj{[0,M]}{C} - A_{f} D F (\bar{x}_{\sC})\left(\proj{}{\mathbb{R}} + \proj{[0,3M+1]}{C} \right)
        -
        A_{f}
        DB(\bar{x}_{\sC})  \proj{[0,3M+1]}{C} -
        A_{f}  D\bar{B}(\bar{x}_{\sC}) \proj{}{\mathbb{R}}
        \rVert_{B(X_{\sC})}
        \\& \quad+
        \frac{ r_{\sTF}  \lVert A_{f} \rVert_{B(X_{\sC})}}{(1- |\bar{\sigma}|)^2 }
        +
        \frac{\omega|\kappa|}{M}
        \max \left\{ 4, |a| + |b|\lVert \bar{s}^{(1)} \lVert_{\sC} + |b|\lVert \bar{s}^{(3)} \lVert_{\sC}  + 3|c|\lVert \bar{s}^{(1)}*\bar{s}^{(1)}  \lVert_{\sC} \right\}
        +
        \frac{ \lVert A_{f}
            \rVert_{B(X_{\sC})} }{\omega^{3M+2}} .
        \nonumber
    \end{align}
\end{lemma}

\begin{proof}
    In this case, the derivate of the validation map is given by:

    \begin{equation} 
        DF(\bar{x}_{\sC}) = \begin{pmatrix} 0 & 0\\ 0 &  L  + \kappa \ TDf (\bar{s})   \end{pmatrix}   + DB(\bar{x}_{\sF}) .
    \end{equation}
    \[
        D f (\bar{s})s =
        \begin{pmatrix}
            s^{(2)}                                                           \\
            -  as^{(1)}
            +  b\left(\bar{s}^{(3)}*_{\sC}s^{(1)}\right)
            +  b\left(s^{(3)}*_{\sC}\bar{s}^{(1)}\right)
            +  3c\left(\bar{s}^{(1)}*_{\sC}\bar{s}^{(1)}*_{\sC}s^{(1)}\right) \\
            s^{(4)}                                                           \\
            -4s^{(3)}
        \end{pmatrix}.
    \]
    Analogously to the previous sections we have
    \begin{equation}
        \label{eq:soliton_bound_Df}
        \lVert Df(\bar{s}) \rVert_{B(X_{\sC})} \leq
        \max \left\{ 4, |a| + |b|\lVert \bar{s}^{(1)} \lVert_{\sC} + |b|\lVert \bar{s}^{(3)} \lVert_{\sC}  + 3|c|\lVert \bar{s}^{(1)}*\bar{s}^{(1)}  \lVert_{\sC}   \right\}.
    \end{equation}
    The term \(D B(\bar{x}_{\sC})x\) is nonzero only in the parameter space component and the first coefficients of each sequence variable. Indeed, for \(x = (\sigma, s) \in X_{\sC}\) we have
    \[
        \proj{}{\mathbb{R}}  DB(\bar{x}_{\sC})x =
        s^{(2)}(-1),
    \]
    \[
        \proj{[0,\infty)}{C} DB(\bar{x}_{\sC}) x=
        \proj{\{0\}}{C} DB(\bar{x}_{\sC}) x=
        \left(
        s^{(1)}(1) - \sigma \frac{\partial}{\partial \sigma} W^{(1)}(\theta,\bar{\sigma}) ,
        s^{(2)}(1) - \sigma \frac{\partial}{\partial \sigma} W^{(2)}(\theta,\bar{\sigma}),
        s^{(3)}(-1)   ,
        s^{(4)}(-1)
        \right)
        .
    \]
    To find a bound for $Z_1$, we split the action of the operator $I - ADF(\bar{x}_{\sC})$ as follows
    \begin{align}
        \label{eq:soliton_splliting}
        \lVert  I - A D F (\bar{x}_{\sC})  \rVert_{B(X_{\sC})}
         & \leq
        \lVert  \proj{}{\mathbb{R}} + \proj{[0, M]}{C}  - A_{f} D F (\bar{x}_{\sC}) \left(\proj{}{\mathbb{R}}  + \proj{[0,3M+1]}{C}\right) \rVert_{B(X_{\sF})} +
        \lVert A_{f} D F (\bar{x}_{\sC})  \proj{(3M+1,\infty)}{C} \rVert_{B(X_{\sC})}
        \\  & +
        \lVert  \proj{(M, 3M+1]}{C} -A_{\infty} D F (\bar{x}_{\sC}) \left(\proj{}{\mathbb{R}}  + \proj{[0,3M+1]}{C}\right)
        +
        \proj{(3M+1,\infty)}{C} - A_{\infty} D F (\bar{x}_{\sC}) \proj{(3M+1, \infty)}{C}\rVert_{B(X_{\sC})}.
        \nonumber
    \end{align}
    We now present computable bounds for each term after the inequality in  \eqref {eq:soliton_splliting}. For the first term, by definition we have
    \[
        A_{f} D F (\bar{x}_{\sC}) \left(\proj{}{\mathbb{R}}  + \proj{[0,3M+1]}{C}\right)
        = A_{f} \begin{pmatrix} 0 & 0\\ 0 &  \proj{[0,M]}{C}  + \kappa \ TDf (\bar{s}) \proj{[0,3M+1]}{C}  \end{pmatrix}
        +
        A_{f} DB(\bar{x}_{\sC}) \left(\proj{}{\mathbb{R}}  + \proj{[0,3M+1]}{C}\right) .
    \]
    Observe that the boundary term can be written as follows
    \[
        A_{f} DB(\bar{x}_{\sC}) \left(\proj{}{\mathbb{R}}  + \proj{[0,3M+1]}{C}\right)
        =
        A_{f} DB(\bar{x}_{\sC}) \proj{[0,3M+1]}{C}
        +
        A_{f} D\bar{B}(\bar{x}_{\sC}) \proj{}{\mathbb{R}}
        +
        A_{f} \left( D\bar{B}(\bar{x}_{\sC}) - DB(\bar{x}_{\sC}) \right) \proj{}{\mathbb{R}} .
    \]
    Where the term $D\bar{B}(\bar{x}_{\sC})\proj{}{\mathbb{R}}$ is defined for $x\in X_{\sC} $ as
    \[
        \proj{}{\mathbb{R}}  D\bar{B}(\bar{x}_{\sC})\proj{}{\mathbb{R}}x =
        0,
        \quad
        \proj{[0,\infty)}{C} D\bar{B}(\bar{x}_{\sC}) \proj{}{\mathbb{R}}x=
        \proj{\{0\}}{C} D\bar{B}(\bar{x}_{\sC})\proj{}{\mathbb{R}}x=
        -\sigma \left(
        \frac{\partial}{\partial \sigma} \bar{W}^{(1)}(\theta,\bar{\sigma}) ,
        \frac{\partial}{\partial \sigma} \bar{W}^{(2)}(\theta,\bar{\sigma}),
        0   ,
        0
        \right)
        .
    \]
    For the next term we apply an analogous result to Lemma \ref{lemma:convolution_product}. In this case,  we consider a truncation up to mode $3M+1$ to account for the action of the operator $T$, which shifts the non-zero modes down by one. Indeed
    \[
        T D f(\bar{s}) \proj{(3M+1,\infty)}{C}
        =
        T  \proj{(M+1,\infty)}{C}  D f(\bar{s}) \proj{(3M+1,\infty)}{C}
        =
        \proj{[M+1,\infty)}{C} T  \proj{[M+1,\infty)}{C}  D f(\bar{s}) \proj{(3M+1,\infty)}{C} .
    \]
    Therefore, the finite part of the operator evaluated at the infinite tail is given by
    \begin{equation*}
        \begin{aligned}
            A_{f} D F(\bar{x}_{\sC}) \proj{(3M+1,\infty)}{C} & =
            A_{f}\begin{pmatrix} 0 & 0\\ 0 &  L  + \kappa \ TDf (\bar{s})   \end{pmatrix} \proj{(3M+1,\infty)}{C}  +A_{f} DB(\bar{x}_{\sF}) \proj{(3M+1,\infty)}{C}
            \\ &=
            A_{f}\begin{pmatrix} 0 & 0\\ 0 &   \kappa \proj{[0,M]}{C}\ TDf (\bar{s}) \proj{(3M+1,\infty)}{C}  \end{pmatrix}  +A_{f} DB(\bar{x}_{\sF}) \proj{(3M+1,\infty)}{C}
            \\  &=
            A_{f} D B(\bar{x}_{\sC})\proj{(3M+1,\infty)}{C}.
        \end{aligned}
    \end{equation*}
    We continue with the third term in equation \eqref{eq:soliton_splliting}. For the evaluation of the infinite part of the operator at the finite tail is straightforward to show that
    \begin{equation*}
        \begin{aligned}
            \proj{(M,3M+1]}{C}  - A_{\infty} D F (\bar{x}_{\sC})\left(\proj{}{\mathbb{R}} + \proj{[0,3M+1]}{C} \right) & =
            - \ \kappa \ L^{-1} \proj{(M,\infty)}{C} T D f (\bar{s})\proj{[0,3M+1]}{C}
            - L^{-1} \proj{(M,\infty)}{C}D B (\bar{x}_{\sC})\left(\proj{}{\mathbb{R}} + \proj{[0,3M+1]}{C} \right)
            .
        \end{aligned}
    \end{equation*}
    The second term in the equation above cancels out. Indeed,
    \[
        L^{-1} \proj{(M,\infty)}{C}D B (\bar{x}_{\sC})\left(\proj{}{\mathbb{R}} + \proj{[0,3M+1]}{C} \right)
        =
        L^{-1} \proj{(M,\infty)}{C}
        \left(\proj{}{\mathbb{R}} + \proj{\{0\}}{C}\right)
        D B (\bar{x}_{\sC})\left(\proj{}{\mathbb{R}} + \proj{[0,3M+1]}{C} \right)=(0,0).
    \]
    Similarly, the infinite part of the operator evaluated at the infinite tail is given by
    \begin{equation*}
        \begin{aligned}
            \proj{(3M+1,\infty)}{C} -A_{\infty} D F(\bar{x}_{\sC}) \proj{(3M+1,\infty)}{C} & =
            -  \kappa \ L^{-1} \proj{(M,\infty)}{C} T D f(\bar{s})\proj{(3M+1,\infty)}{C}
        \end{aligned}.
    \end{equation*}
    After putting all the terms together and using the triangle inequality we obtain the following bound
    \begin{equation*}
        \begin{aligned}
            \lVert I - ADF(\bar{x}_{\sC})  \rVert_{B(X_{\sC})} & \leq
            \lVert
            \proj{}{\mathbb{R}} \ + \proj{[0,M]}{C} - A_{f} \begin{pmatrix} 1 & 0\\ 0 &  \proj{[0,M]}{C}  + \kappa \ TDf (\bar{s}) \proj{[0,3M+1]}{C}  \end{pmatrix}
            \\&
            \quad +
            A_{f}
            DB(\bar{x}_{\sC})  \proj{[0,3M+1]}{C}
            +
            A_{f}  D\bar{B}(\bar{x}_{\sC}) \proj{}{\mathbb{R}}
            \rVert_{B(X_{\sC})}
            \\&
            \quad +
            \lVert A_{f}
            \rVert_{B(X_{\sC})}
            \lVert  \left( D\bar{B}(\bar{x}_{\sC}) - DB(\bar{x}_{\sC}) \right) \proj{}{\mathbb{R}}\rVert_{B(X_{\sC})}
            \\&
            \quad +
            |\kappa|
            \lVert L^{-1}
            \proj{(M,\infty)}{C} \rVert_{B(X_{\sC})}\lVert T \rVert_{B(X_{\sC})}
            \lVert Df(\bar{s}) \rVert_{B(X_{\sC})}
            \\&
            \quad +   \lVert A_{f}
            \rVert_{B(X_{\sC})}
            \lVert D B(\bar{x}_{\sC})\proj{(3M+1,\infty)}{C}\rVert_{B(X_{\sC})} .
        \end{aligned}
    \end{equation*}
    The first term after the inequality are already computable. For the terms involving the evaluation at the parameter space component of \(D B(\bar{x}_{\sC})\),  for any $x=(\sigma,s)\in X_{\sC}$ such that $\lVert x \rVert_{X_{\sC}}\leq1$ we have
    \begin{equation}
        \begin{aligned}
            \lVert \left( D\bar{B}(\bar{x}_{\sC}) - DB(\bar{x}_{\sC}) \right) \proj{}{\mathbb{R}}  x\rVert_{X_{\sC}} & \leq
            \max_{i=1,2}\left\{
            \left|
            \frac{\partial}{\partial \sigma} W^{(i)}(\theta,\bar{\sigma})
            -
            \frac{\partial}{\partial \sigma} \bar{W}^{(i)}(\theta,\bar{\sigma})
            \right|
            \right\}
        \end{aligned}
    \end{equation}
    In order to bound the term involving the derivative we use \eqref{eq:TF_expansion} to obtain
    \[
        \frac{\partial  }{\partial \sigma }W^{(i)}(\theta,\bar{\sigma}) = \sum_{n \in \N}
        (n+1)  W_{n+1}^{(i)}
        \bar{\sigma}^{n}.
    \]
    Since we a supposing that the paramaterization $W$ of the stable manifold is obtained as in Theorem \ref{Theorem:Manifold}, we know there exists a $\delta \in X_{\sTF}$ such that
    \[
        w = \bar{w} + \delta
        \quad
        \text{and}
        \quad
        \left\| \delta ^{(i)}\right\|_{\sTF }
        =
        \sum_{n\geq 0} \left\| \delta^{(i)}_{n} \right\|_{\sF }
        \leq r_{\sTF}
    \]
    Moreover, we are supposing that  $-1<\bar{\sigma}<1$ and $1\leq\nu$, hence
    \[
        \left|
        \frac{\partial}{\partial \sigma} W^{(i)}(\theta,\bar{\sigma})
        -
        \frac{\partial}{\partial \sigma} \bar{W}^{(i)}(\theta,\bar{\sigma})
        \right|
        \leq
        \left|
        \sum_{n \in \N}
        (n+1)
        \bar{\sigma}^{n} \delta_{n+1}^{(i)}
        \right|
        \leq
        \left\| \delta^{(i)}_n \right\|_{\sF}
        \sum_{n=0}^{\infty} (n+1)|\bar{\sigma}|^{n}
        \leq \frac{r_{\sTF}}{(1- |\bar{\sigma}|)^2 }.
    \]
    Then, we can conclude that
    \[
        \lVert  \left( D\bar{B}(\bar{x}_{\sC}) - DB(\bar{x}_{\sC}) \right) \proj{}{\mathbb{R}} \rVert_{B(X_{\sC})}  \leq
        \frac{r_{\sTF}}{(1- |\bar{\sigma}|)^2 }
        .
    \]
    The third term can be bounded using equation \eqref{eq:soliton_bound_Df} together with the bounds
    \[
        \lVert T \rVert_{B(X_{\sC})} \leq 2\omega,
        \quad
        \lVert L^{-1} \proj{(M,\infty)}{C}\rVert_{B(X_{\sC})}  \leq \frac{1}{2M}.
    \]
    Finally, for the evaluation at the infinite tail, first remember that  $\omega \geq 1$. Hence, it holds that
    \[
        2\sum_{m > 3M+1} |s_{m}^{(i)}| = 2\sum_{m > 3M+1} \frac{|s_{m}^{(i)}|\omega^m}{\omega^m} \leq \frac{\lVert s^{(i)} \rVert_{\sC}}{\omega^{3M+2}}\leq\frac{1}{\omega^{3M+2}}.
    \]
    for $i=1,2,3,4$ and therefore
    \[
        \lVert D B(\bar{x}_{\sC})\proj{(3M+1,\infty)}{C} x \rVert_{X_{\sC}} = \max_{i=1,2,3,4} \left \{2\sum_{m > 3M+1}  |s_{m}^{(i)}|  \right \}
        \leq
        \frac{1}{\omega^{3M+2}}.
    \]
    We have shown that a computable $Z_1$ bound is given by equation \eqref{eq:soliton_Z1}.
\end{proof}

\begin{lemma}
    \label{lemma:soliton_Z2}
    Let $r\in \mathbb{R}$ such that $|\bar{\sigma}+r| < 1$. A computable $Z_{2}(\bar{x}_{\sC}, r_{\sTF})$  satisfying  $\lVert A(DF(x) - DF(\bar{x}_{\sC})) \rVert_{B(X_{\sC})} \leq Z_{2}(\bar{x}_{\sC}, r_{\sTF}) \lVert x - \bar{x}_{\sC} \rVert_{X_{\sC}} $ for all $x$ in $B(\bar{x}_{\sC}, r)$ is given by

    \begin{align}\label{eq:soliton_Z2}
        Z_{2}(\bar{x}_{\sC}, r_{\sTF}) & \bydef
        2\omega|\kappa|
        \left(
        \lVert A_{f} \rVert_{B(X_{\sC})}
        +
        \frac{1}{2M}
        \right)
        \left( 2|b| + 6|c|\lVert \bar{s}^{(1)}_{\sC} \rVert_{\sC} + 3|c|r \right )
        \\ &
        +
        \lVert A_{f} \left(\proj{}{\mathbb{R}}  + \proj{\{0\}}{C}\right) \rVert_{B(X_{\sC})} \max_{i=1,2} \left\{
        \sup_{\sigma \in B_{\R}(\bar{\sigma},r)}
        \frac{\left(|\bar{\sigma}|^2 + |\bar{\sigma}|\right)r_{\sTF} }{(1- |\bar{\sigma}|)^3 }
        +
        \frac{|\bar{\sigma}|r_{\sTF} + 2r_{\sTF}}{(1- |\bar{\sigma}|)^2 }
        +
        \left | \frac{\partial^2  }{\partial \sigma ^2 }\bar{W}^{(i)}(\theta,\sigma) \right |
        \right\}.
        \nonumber
    \end{align}
\end{lemma}

\begin{proof}
    Observe that for any $ x =(\sigma_x , s_x) \in B(\bar{x}_{\sC}, r) $ we have,
    \[
        \lVert  DF(x) - DF(\bar{x}_{\sC})  \rVert_{B(X_{\sC})}
        \leq
        |\kappa|
        \lVert T \rVert_{B(X_{\sC})}  \lVert  Df(s_x) - Df(\bar{s})  \rVert_{B(X_{\sC})}
        +
        \lVert  DB(x) - DB(\bar{x}_{\sC})  \rVert_{B(X_{\sC})}.
    \]
    For the terms not involving the parameterization of the stable manifiold we proceed as in our previous proofs:
    \[
        \lVert  Df(x) - Df(\bar{x}_{\sC}) \rVert_{B(X_{\sC})} \leq r \left ( 2|b| + 6|c|\lVert \bar{s}^{(1)} \rVert_{\sC} + 3|c|r \right ).
    \]
    For the term including the boundary conditions. We begin by considering an element $h=(\sigma_{h}, s_{h})$ of $X_{\sC}$ such that $\lVert h\lVert_{X_{\sC}} \leq 1 $ . We have that
    \begin{equation*}
        \begin{aligned}
            \lVert \left(  DB(x) - DB(\bar{x}_{\sC}) \right) h \rVert_{X_{\sC}}
             & =
            \lVert \proj{\{0\}}{C} \left( DB(x) - DB(\bar{x}_{\sC}) \right)h \rVert_{X_{\sC}}
            \\  & \leq
            |\sigma_{h}| \max_{i=1,2} \left\{
            \left|
            \frac{\partial  }{\partial \sigma}W^{(i)}(\theta,\sigma_{x})
            -
            \frac{\partial  }{\partial \sigma}W^{(i)}(\theta,\bar{\sigma})
            \right|
            \right\}
            \\  & \leq
            r \max_{i=1,2} \left\{
            \left|
            \frac{\partial^2  }{\partial \sigma ^2 }W^{(i)}(\theta, d_i)
            \right|
            \right\},
        \end{aligned}
    \end{equation*}
    for some  $d_1, d_2$ in $B_{\R}(\bar{\sigma},r)$, where we have use the  mean value inequality. To compute the second derivatives in the previous expression we proceed as in the proof of Lemma \ref{lemma:soliton_Z1}. For $i = 1,2$,  we know that
    \[
        \frac{\partial^2  }{\partial \sigma ^2 }W^{(i)}(\theta,\sigma) = \sum_{n \in \N}
        (n+2)(n+1)  W_{n+2}^{(i)}
        \sigma^{n}.
    \]
    Hence, for any  $|\sigma|<1$  in $B_{\R}(\bar{\sigma},r)$  since $|\bar{\sigma}+r| < 1$, we have
    \[
        \sup_{\sigma \in B_{\R}(\bar{\sigma},r)}
        \left | \frac{\partial^2  }{\partial \sigma ^2 }W^{(i)}(\theta,\sigma) \right | \leq
        \sup_{\sigma \in B_{\R}(\bar{\sigma},r)}
        \frac{\left(|\sigma|^2 + |\sigma|\right)r_{\sTF} }{(1- |\sigma|)^3 }
        +
        \frac{|\sigma|r_{\sTF}}{(1- |\sigma|)^2 }
        +
        \frac{2r_{\sTF}}{(1- |\sigma|)^2 }+
        \left |\frac{\partial^2  }{\partial \sigma ^2 }\bar{W}^{(i)}(\theta,\sigma) \right | .
    \]
    The above follows from the fact that
    \[
        \sum_{n=0}^{\infty} n|\sigma|^{n} = \frac{|\sigma|}{(1- |\sigma|)^2 },
        \quad
        \sum_{n=0}^{\infty} n^2|\sigma|^{n} = \frac{|\sigma|^2 + |\sigma|}{(1- |\sigma|)^3 }.
    \]
    We note that finding the supremum of a computable formula over a ball is equivalent to taking the upper bound of the interval evaluation of the expression. Thus, a computable bound for \(Z_2\) is given by \eqref{eq:soliton_Z2}.
\end{proof}
The following Theorem gives us the computable conditions to check that close to an approximate zero of sequence equation \eqref{eq:soliton_F} there exists a true zero.

\begin{theorem}
    \label{Theorem:Soliton}
    Fix parameters \(a\), \(b\), and \(c\), a weight \(\nu\) for the norm in \eqref{eq:S_F}. Suppose $W$ is a parameterization for the stable manifold attached to the periodic orbit $\gamma$ with radius $r_{\sTF}$ as obtained by Theorem \ref{Theorem:Manifold}.  Let $\bar{x}_{\sC}= (\bar{\sigma}, \bar{s}) \in X_{\sC}$ such that each component of $\bar{s}$ is supported in $[0,M]$.
    Additionally, assume that the matrix \(\bar{A}(\bar{x}_{\sC})\) from operator \eqref{eq:bundle_A} is computed as the numerical inverse of
    \[
        \left( \proj{}{\mathbb{C}} + \proj{[0, M]}{C} \right)
        DF(\bar{x}_{\sC}) \left( \proj{}{\mathbb{C}} + \proj{[0, M]}{C} \right).
    \]
    Suppose the computable bounds \(Y(\bar{x}_{\sC}, r_{\sTF})\), \(Z_1(\bar{w}, \bar{x}_{\sF})\), and \(Z_{2}(\bar{x}_{\sC}, r_{\sTF}^{*})\) defined in \eqref{eq:soliton_Y}, \eqref{eq:soliton_Z1} and \eqref{eq:soliton_Z2} satisfy
    \[    Z_{1}(\bar{x}_{\sC}, r_{\sTF}) < 1 \quad \text{and} \quad Z_{2}(\bar{x}_{\sC}, r_{\sC}^{*}) < \frac{(1 - Z_{1}(\bar{x}_{\sC}, r_{\sTF}))^{2}}{2Y(\bar{x}_{\sC}, r_{\sTF})}
    \]
    for some $r_{\sC}^{*}$.  Then, the validation map \eqref{eq:soliton_F} has a unique zero  $x_{\sC} = (\sigma, s)$ in the ball $B_{X_{\sC}}(\bar{x}_{\sC}, r_{\sC})$  such that $F(x_{\sC}) =0$,
    where the radius of the ball is given by:
    \[
        r_{\sC}=\frac{1 - Z_{1}(\bar{x}_{\sC}, r_{\sTF})- \sqrt{ (1- Z_{1}(\bar{x}_{\sC}, r_{\sC}^{*}))^2 - 2Y(\bar{x}_{\sC}, r_{\sTF})Z_{2}(\bar{x}_{\sC}, r_{\sC}^{*})}}{Z_{2}(\bar{x}_{\sC}, r_{\sC}^{*})}.
    \]
\end{theorem}
\begin{proof}
    The result follows from Theorem \ref{Theorem:Newton}.
\end{proof}
We are now ready to validate numerically approximated soliton solutions.


\section{Examples of Constructive Proofs of Existence of Gap Solitons} \label{sec:conclusions}

In this section, we present examples of our computer-assisted method for proving the existence of soliton solutions to equation \eqref{eq:GP_ode}. Using Theorems \ref{Theorem:Bundle}, \ref{Theorem:Manifold}, and \ref{Theorem:Soliton}, we demonstrate that true soliton solutions exist near numerically approximated ones. 

Given a numerical approximation to a soliton solution, our approach requires implementing all relevant expressions—operators, derivatives, and bounds—introduced in earlier sections. The computer-assisted component comes in the form of rigorously evaluating these quantities using interval arithmetic. For our examples, we use the Julia programming language and we use interval arithmetic through the \texttt{IntervalArithmetic.jl} Julia package \cite{RadiiPolynomial.jl}. Our code implementation is available in the repository associated with this paper \cite{github_codes}.

To verify all bounds in the hypotheses of Theorems~\ref{Theorem:Bundle}, \ref{Theorem:Manifold}, and \ref{Theorem:Soliton}, we evaluate each computable quantity using intervals rather than individual floating-point numbers. This guarantees that the true value lies within the resulting interval. The right endpoint provides a rigorous upper bound for the quantity being estimated. As a first example, we restate Theorem~\ref{theorem:main} and provide a computer-assisted proof.

\begin{theorem}
    The Gross-Pitaevskii equation \eqref{eq:GP_ode} with parameters \( a = 1.1025 \), \( b = 0.55125 \), and \( c = -0.826875 \) has a soliton solution $u:\R \to \R$,
    satisfying
    \[
    \|u - \bar{u}\|_{\infty} \leq 8.617584260554394 \cdot 10^{-6},
    \]
    where \( \bar{u} \) is a numerical approximation of the solution illustrated in Figure~\ref{fig:Parts_of_proof}. 
\end{theorem}

\begin{proof}
    We consider computable sequences \(\bar{x}_{\sF} \in X_{\sF}\), \(\bar{x}_{\sTF} \in X_{\sTF}\), and \(\bar{x}_{\sC} \in X_{\sC}\), that approximate zeros of \eqref{eq:bundle_F}, \eqref{eq:manifold_F}, and \eqref{eq:soliton_F}, respectively
    \begin{equation}
    \label{eq:truncations}
        \bar{x}_{\sF}= \proj{[0,M_{\sF}]}{\sF}\bar{x}_{\sF},
        \quad
        \bar{x}_{\sTF}  = \proj{[0,N_{\sT}]}{T}\proj{[0,M_{\sF}]}{F}\bar{x}_{\sTF},
        \quad
        \bar{x}_{\sC} = \proj{[0,M_{\sC}]}{C}\bar{x}_{\sC}.
    \end{equation}
We choose truncation modes \(M_{\sF}\), \(N_{\sT}\), and \(M_{\sC}\) to capture the nonlinear behavior of the equation, so that the final coefficients of each sequence have decayed to the level of rounding error in double precision.
  For this particular set of parameters, we have obtained \(\bar{x}_{\sF} \in X_{\sF}\), \(\bar{x}_{\sTF} \in X_{\sTF}\), and \(\bar{x}_{\sC} \in X_{\sC}\) satisfying \eqref{eq:truncations} for $M_{\sF} = 32$,  $N_{\sT} = 32$  and $M_{\sC} = 48$. Together with the following residual conditions
\begin{equation}
    \label{eq:residual}
    \lVert  F (\bar{x}_{\sF}) \rVert_{X_{\sF}} ,
    \quad
    \lVert  F (\bar{x}_{\sTF}) \rVert_{X_{\sTF}} ,
    \quad
    \lVert  F (\bar{x}_{\sC}) \rVert_{X_{\sC}} \leq 10^{-13}.
\end{equation} 
    For the Fourier expansion, we use the same truncation mode for both the bundle and the manifold, as we can always take the maximum between the two when needed.  The coefficients used in these computations are available in the repository associated with this paper \cite{github_codes}.  In our implementation, we use the \texttt{RadiiPolynomial.jl} Julia package \cite{RadiiPolynomial.jl} to easily manipulate sequences.

The required level of numerical accuracy depends on the solution we aim to validate. However, as a byproduct of our approach, the zero-finding problems \eqref{eq:bundle_F}, \eqref{eq:manifold_F}, and \eqref{eq:soliton_F} define well-conditioned maps that can be used with root-finding algorithms to refine our numerical approximations. In practice, after computing an initial approximation, we apply Newton’s method to these maps to improve accuracy to the required level.
 
Our proof has three stages. We begin with the bundle problem, whose solution provides the input needed to evaluate the validation map for the stable manifold attached to the periodic orbit defined by \eqref{eq:bundle_F}. Next, we construct and validate the parameterization of this manifold. From this, we obtain an explicit boundary condition for the boundary-value problem \eqref{eq:BVP_soliton}, which we then use to validate a numerical approximation of the soliton solution.

For the first stage, the numerical approximation of the bundle problem \(\bar{x}_{\sF} = (\bar{\lambda}, \bar{v}) \in X_{\sF}\) satisfies \(\bar{\lambda} < 0\). Moreover, the coefficients of \(\bar{v}\) satisfy the symmetry condition  
\begin{equation}
    \bar{v}_k^{(i)} = [\bar{v}_{-k}^{(i)}]^*, \quad \text{for} \quad i = 1, 2, \quad k \in \mathbb{Z}.
\end{equation}  
Note that this condition can be directly imposed and verified in the computational implementation. We fix the norm weight \(\nu = 1.05\) in \eqref{eq:S_F} and the scaling factor \(l = 0.5\) in the phase condition \eqref{eq:phase_conditon}. Using interval arithmetic, we validate the following bounds:
\[
Y(\bar{x}_{\sF}) = 2.6879100002352747 \cdot 10^{-13}, \quad
Z_1(\bar{x}_{\sF}) = 0.3465291783592818, \quad
Z_2(\bar{x}_{\sF}) = 14.980732463866438,
\]
as defined in \eqref{eq:bundle_Y}, \eqref{eq:bundle_Z1}, and \eqref{eq:bundle_Z2}. These bounds satisfy the inequalities
\[
Z_{1}(\bar{x}_{\sF}) < 1 \quad \text{and} \quad Z_{2}(\bar{x}_{\sF}) < \frac{(1 - Z_{1}(\bar{x}_{\sF}))^{2}}{2Y(\bar{x}_{\sF})},
\]
which meet the hypotheses of Theorem~\ref{Theorem:Bundle}. We therefore obtain a unique zero of the validation map \eqref{eq:bundle_F},
\[
x_{\sF} \bydef (\lambda, v^{(1)}, v^{(2)}),
\]
within the ball \( B_{X_{\sF}}(\bar{x}_{\sF}, r_{\sF}) \) of radius \( r_{\sF} = 4.122891017172993 \cdot 10^{-13} \). This zero defines a solution
\(
(\lambda, v^{(1)}, v^{(2)}, 0, 0) \in \mathbb{R} \times S^{4}_{\sF},
\)
to the sequence equation \eqref{eq:bundle_coeff} and corresponds to a real solution of the bundle differential equation \eqref{eq:bundle_ODE}.

The second stage of our proof corresponds to the construction of the parameterization of the stable manifold. In this case, for $r_{\sTF}^{*}=10^{-3}$ , the bounds  \(Y(\bar{w}, r_{\sF})=6.327932449800631 \cdot 10^{-9}\), \(Z_1(\bar{w}, \bar{x}_{\sF})=0.9583731072113382\), and \(Z_2(\bar{w}, r_{\sTF}^{*})=104.77593347038471\) defined in \eqref{eq:manifold_Y}, \eqref{eq:manifold_Z1}, and \eqref{eq:manifold_Z2} satisfy
    \[
        Z_{1}(\bar{w}, \bar{x}_{\sF})<  1 \quad \text{and} \quad Z_2(\bar{w}, r_{\sTF}^{*})  < \frac{(1 - Z_{1}(\bar{w}, \bar{x}_{\sF}))^{2}}{2Y(\bar{w}, r_{\sF})}.
    \]
    Therefore, we use Theorem \ref{Theorem:Manifold} to validate our parameterization of the manifold $x_{\sTF}\in X_{\sTF}$ with a radius \(r_{\sTF}=1.5204458252945915\cdot 10^{-7}\). The resulting stable manifold is represented in orange in Figure \ref{fig:Parts_of_proof}.

The third and final stage is to validate a numerical approximation of the corresponding boundary-value problem. For this case, we set \(\theta = 1\) and \(\Lr = 1 + 2\pi\) in \eqref{eq:BVP}. The value for the Taylor variable is
\(
\bar{\sigma} = 0.927447198734628.
\)
As in the previous stages, we fix a norm weight \(\omega = 1.05\) in \eqref{eq:S_F} and compute the following bounds:
\[
Y(\bar{x}_{\sC}, r_{\sTF}) = 7.814019760054922 \cdot 10^{-7}, \quad Z_1(\bar{w}, \bar{x}_{\sF}) = 0.9076283031949424 \quad
\]
\[
\text{and} \quad Z_2(\bar{x}_{\sC}, r_{\sTF}^{*}) = 372.96640912543626,
\]
as defined in \eqref{eq:soliton_Y}, \eqref{eq:soliton_Z1}, and \eqref{eq:soliton_Z2}. These bounds satisfy
\[
Z_{1}(\bar{x}_{\sC}, r_{\sTF}) < 1 \quad \text{and} \quad Z_{2}(\bar{x}_{\sC}, r_{\sC}^{*}) <  \frac{\bigl(1 - Z_{1}(\bar{x}_{\sC}, r_{\sTF})\bigr)^{2}}{2Y(\bar{x}_{\sC}, r_{\sTF})} 
\]
for \(r_{\sC}^{*} = 10^{-2}\). Notice that \(|\bar{\sigma} + r_{\sC}^{*}| < 1\), as required by Lemma~\ref{lemma:soliton_Z2}. We then apply Theorem~\ref{Theorem:Soliton} to obtain a solution \(x = (\sigma, s) \in X_{\sC}\) such that
\[
s^{(i)}(t) = s^{(i)}_0 + 2 \sum_{m\geq 1} s^{(i)}_m T_m(t), \quad i = 1,2,3,4,
\]
solves the boundary-value problem \eqref{eq:BVP_soliton} (after reverting the scaling in time) in the ball centered at \(\bar{x} = (\bar{\sigma}, \bar{s})\) with radius 
\(
r_{\sC} = 8.617584260554394 \cdot 10^{-6}.
\)
The solution of the boundary-value problem is depicted in blue in Figure~\ref{fig:Parts_of_proof}. 

As described in the introduction, by taking the even extension of a solution to the boundary-value problem \eqref{eq:BVP_soliton}, we obtain a soliton solution to the Gross-Pitaevskii equation \eqref{eq:GP_ode}. The proven soliton solution satisfies
    \[
        \|\bar{s}^{(i)} - s^{(i)}\|_{\infty}= \sup_{t\in[-1,1]} \|\bar{s}^{(i)}(t) - s^{(i)}(t)\|_{\sC} \leq r_{\sC} =8.617584260554394 \cdot 10^{-6} , \quad i = 1,2,3,4. 
    \]
    Observe that, by \eqref{eq:parameterization_bound}, the error bound also applies to the portion of the solution obtained via the parameterization of the stable manifold, not just the segment in the boundary-value problem region. The proofs with interval arithmetic of all the inequalities above are included in \cite{github_codes}.   
\end{proof}

The implementation of our method can be easily adapted to different sets of parameters. For example, we provide existence proofs for the solitons presented in \cite{Alfimov_2002}, including the one- and two-hump solutions illustrated in Figure~\ref{fig:other_proofs}.

\begin{figure}[h!]
    \centering
    \begin{subfigure}[b]{0.49\textwidth}
        \centering
        \includegraphics[width=\textwidth]{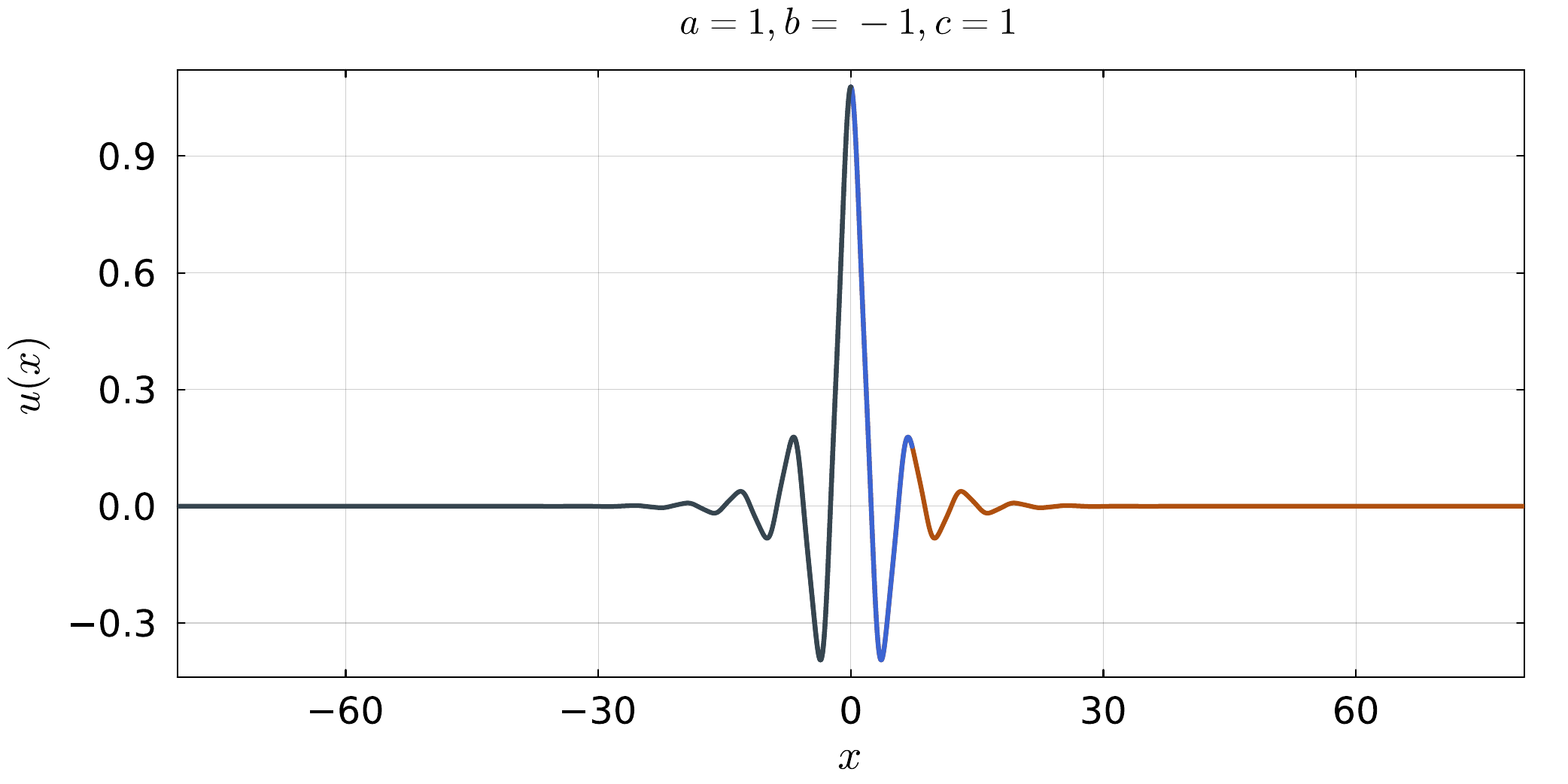}
    \end{subfigure}
    \hfill
    \begin{subfigure}[b]{0.49\textwidth}
        \centering
        \includegraphics[width=\textwidth]{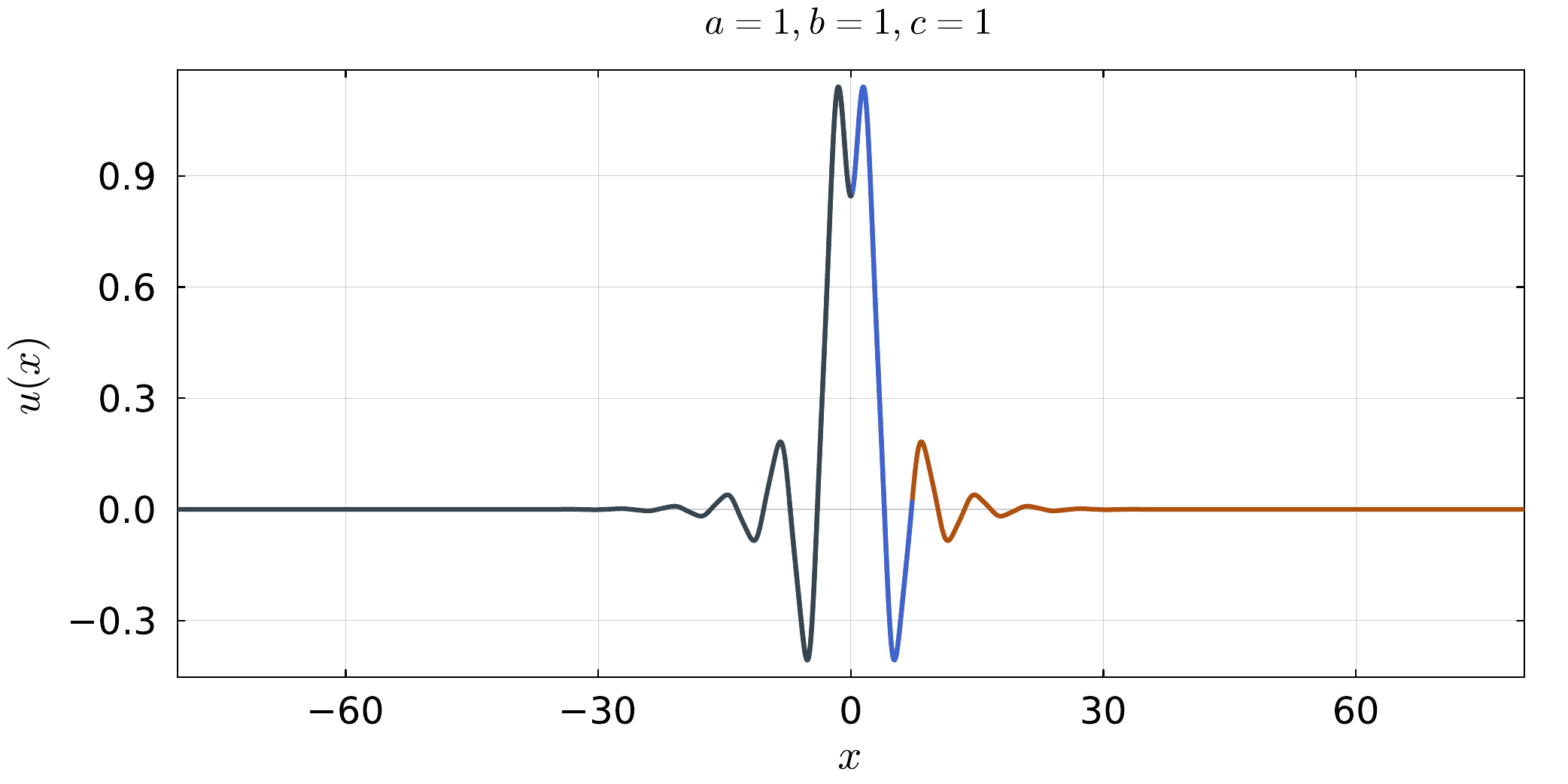}
    \end{subfigure}
    \caption{
Computer-assisted proofs for numerically approximated solitons originally presented in \cite{Alfimov_2002}. The parameters of the equation are shown at the top of the figure. The illustration also shows the main components of our approach: the solution to the boundary-value problem (blue), the stable manifold (orange), and the even extension of the soliton (black).
 }%
    \label{fig:other_proofs}
\end{figure}

\begin{theorem}
    The Gross-Pitaevskii equation \eqref{eq:GP_ode} with parameters $a=1$, $b=-1$ and $c=1$, has a soliton solution.
    \label{theorem:second_solution}
\end{theorem}
\begin{proof}
The proof proceeds analogously to that presented in Theorem~\ref{theorem:main}. For this proof, we set the truncation modes as follows
\[
M_{\sF} = 30,\quad N_{\sT} = 30,\quad M_{\sC} = 56.
\]
We fix the norm weight \(\nu = 1.05\) for the norm in \eqref{eq:S_F} and \(\omega = 1.05\) for the norm in \eqref{eq:S_C}. We use a bundle scaling factor \(l = 0.5\).
\end{proof}

\begin{theorem}
    The Gross-Pitaevskii equation \eqref{eq:GP_ode} with parameters $a=1$, $b=1$ and $c=1$, has a soliton solution.
\end{theorem}
\begin{proof}
The proof is analogous to that presented in Theorem~\ref{theorem:main}. We use the same parameters as in Theorem~\ref{theorem:second_solution}.
\end{proof}

\subsubsection*{Author Contributions Statement} All authors equally contributed to this manuscript and should be listed in alphabetical order.

\subsubsection*{Funding} 
Not applicable.



\bibliographystyle{unsrt}
\bibliography{references}

\end{document}